\documentclass[sn-mathphys-num]{sn-jnl}

\usepackage{graphicx}%
\usepackage{multirow}%
\usepackage{amsmath,amssymb,amsfonts}%
\usepackage{amsthm}%
\usepackage{mathrsfs}%
\usepackage[title]{appendix}%
\usepackage{xcolor}%
\usepackage{textcomp}%
\usepackage{manyfoot}%
\usepackage{booktabs}%
\usepackage{algorithm}%
\usepackage{algorithmicx}%
\usepackage{algpseudocode}%
\usepackage{listings}%

\theoremstyle{thmstyleone}%
\newtheorem{theorem}{Theorem}
\newtheorem{corollary}{Corollary}

\theoremstyle{thmstyletwo}%
\newtheorem{example}{Example}%
\newtheorem{remark}{Remark}%

\theoremstyle{thmstylethree}%
\newtheorem{lem}{Lemma}%

\begin{document}

\title[A generalization of the relaxation-based matrix splitting iterative method for solving the system of generalized absolute value equations]{A generalization of the relaxation-based matrix splitting iterative method for solving the system of generalized absolute value equations}


\author[1]{\fnm{Xuehua} \sur{Li}}\email{3222714384@qq.com}

\author*[1]{\fnm{Cairong} \sur{Chen}}\email{cairongchen@fjnu.edu.cn}

\author[2]{\fnm{Deren} \sur{Han}}\email{handr@buaa.edu.cn}

\affil[1]{\orgdiv{School of Mathematics and Statistics \& Key Laboratory of Analytical Mathematics and Applications (Ministry of Education) \& Fujian Provincial Key Laboratory of Statistics and Artificial Intelligence}, \orgname{Fujian Normal University}, \orgaddress{\city{Fuzhou}, \postcode{350117}, \country{China}}}

\affil[2]{\orgdiv{LMIB of the Ministry of Education, School of Mathematical Sciences}, \orgname{Beihang University}, \orgaddress{\city{Beijing}, \postcode{100191}, \country{China}}}



\abstract{By incorporating a new matrix splitting and the momentum acceleration into the relaxed-based matrix splitting (RMS) method \cite{soso2023}, a generalization of the RMS (GRMS) iterative method for solving the generalized absolute value equations (GAVEs) is proposed. On the one hand, unlike some existing methods, by using the Cauchy's convergence principle we give some sufficient conditions for the existence and uniqueness of the solution to GAVEs and prove that our method can converge to the unique solution of GAVEs. On the other hand, we obtain a few new and weaker convergence conditions for some existing methods. Moreover, we establish comparison theorems between GRMS method and some existing methods. Preliminary numerical experiments  show that the proposed method is efficient.}

\keywords{Generalized absolute value equations, Matrix splitting, Relaxation technique, Momentum acceleration, Convergence}


\maketitle

\section{Introduction}\label{sec:intro}
We focus on seeking a solution of the following system of generalized absolute value equations (GAVEs)
\begin{equation}\label{eq:gave}
Ax - B|x| - b = 0,
\end{equation}
where $A,B\in \mathbb{R}^{n\times n}$ and $b\in \mathbb{R}^n$ are given, and $x\in\mathbb{R}^n$ is unknown. Here, $|x|\in\mathbb{R}^n$ denotes the vector whose $i$th component is $|x_i|$, i.e., the absolute value of $x_i$. When~$B=I$, GAVEs~\eqref{eq:gave} reduces to the system of absolute value equations (AVEs)
\begin{equation}\label{eq:ave}
	Ax - |x| - b = 0.
\end{equation}
To the best of our knowledge, GAVEs~\eqref{eq:gave} is formally introduced by Rohn in \cite{rohn2004}. Over the past two decades, GAVEs~\eqref{eq:gave} and AVEs \eqref{eq:ave} have attracted more and more attention in the optimization community and a great deal of research have been done; see, e.g., \cite{mame2006,huhu2010,jizh2013,LCX2016,mang2007,manga2009a,
manga2009b,cyyh2021,lyyhc2023,alct2023,prok2009,chyh2023}. Particularly, solving the general GAVEs~\eqref{eq:gave} is NP-hard \cite{mang2007}.

Within the numerical algorithms for solving GAVEs~\eqref{eq:gave}, some are characterized by using the matrix splitting technique or the relaxation technique. Recently, by using the matrix splitting technique and the relaxation technique, Song and Song \cite{soso2023} present the relaxation-based matrix splitting (RMS) iterative method (see Algorithm \ref{alg:rms} below) for solving GAVEs~\eqref{eq:gave}.

\begin{algorithm}
\caption{}\label{alg:rms}
Assume that $A=M_{\rm RMS}-N_{\rm RMS}$ with $M_{\rm RMS}$ being nonsingular. Given any initial vectors $x^{(0)}, y^{(0)}\in \mathbb{R}^{n}$, for $k=0,1,2,\ldots$ until the iterative sequence $\{(x^{(k)},y^{(k)})\}^\infty_{k=0}$ is convergent, compute
\begin{equation}\label{eq:rms}
\begin{cases}
		x^{(k+1)}=M_{\rm RMS}^{-1}(N_{\rm RMS}x^{(k)} + By^{(k)} +b),\\
        y^{(k+1)}=(1-\tau_{\rm RMS})y^{(k)} + \tau_{\rm RMS}|x^{(k+1)}|,
\end{cases}
\end{equation}
where $\tau_{\rm RMS}$ is a positive constant.
\end{algorithm}

The RMS iterative scheme \eqref{eq:rms} includes many existing iterative schemes as its special cases \cite{soso2023}. For instance,
\begin{enumerate}
  \item[(i)] Let $A = M_{\rm NMS} - N_{\rm NMS}$ and $\Omega_{\rm NMS}$ be a given matrix such that $M_{\rm NMS} + \Omega_{\rm NMS}$ is invertible. If $\tau_{\rm RMS}=1$ and $M_{\rm RMS}=M_{\rm NMS}+\Omega_{\rm NMS}$, the RMS iteration \eqref{eq:rms} with $y^{(0)}= |x^{(0)}|$ turns into the  Newton-based matrix splitting (NMS) iteration \cite{zhwl2021}
      \begin{equation}\label{eq:nms}
      x^{(k+1)} = (M_{\rm NMS} + \Omega_{\rm NMS})^{-1}[(N_{\rm NMS} + \Omega_{\rm NMS})x^{(k)} + B|x^{(k)}| + b].
      \end{equation}
      The NMS iteration \eqref{eq:nms} reduces to the modified Newton-type (MN) iteration~\cite{wacc2019} when $M_{\rm NMS}=A$ and $\Omega_{\rm NMS}$ is a positive semi-definite matrix. Moreover, if $M_{\rm NMS}=A$ and $\Omega_{\rm NMS}=0$, the NMS iteration \eqref{eq:nms} reduces to the Picard
      iteration~\cite{rohf2014}
      \begin{equation}\label{eq:pi}
      x^{(k+1)} = A^{-1}( B|x^{(k)}| + b).
      \end{equation}
      If $M_{\rm NMS} =\dfrac{1}{2}(A+\Omega_{\rm SSMN})$ and $\Omega_{\rm NMS}=0$, then the NMS iteration \eqref{eq:nms} turns into the following shift-splitting modified Newton-type (SSMN) iteration \cite{liyi2021}
  \begin{equation}\label{eq:ssmm}
       x^{(k+1)} = (A+\Omega_{\rm SSMN})^{-1} \left[(\Omega_{\rm SSMN}-A) x^{(k)} + 2B|x^{(k)}| + 2b\right],
      \end{equation}
where $\Omega_{\rm SSMN}$ is a positive semi-definite matrix such that $ A+\Omega_{\rm SSMN} $ is invertible.

  \item[(ii)] If $B=I, M_{\rm RMS}=\dfrac{1}{\omega_{\rm SOR}}A$ and $\tau_{\rm RMS}=\omega_{\rm SOR}>0$, the RMS iteration \eqref{eq:rms} becomes the SOR-like iteration \cite{KM2017}
  \begin{equation}\label{it:sor}
		\begin{cases}
		    x^{(k+1)}=(1-\omega_{\rm SOR})x^{(k)} + \omega_{\rm SOR} A^{-1}\left( y^{(k)} + b\right),\\
			y^{(k+1)}=(1-\omega_{\rm SOR})y^{(k)} + \omega_{\rm SOR} |x^{(k+1)}|.
		\end{cases}
	\end{equation}

  \item[(iii)] Throughout this paper, let $A=D-L-U$ with $D$, $-L$ and $-U$ being the diagonal part, the strictly lower-triangular part and the strictly upper-triangular part of $A$, respectively. If $B=I$, $M_{\rm RMS}=\dfrac{1}{\omega}D-L$ and $\tau_{\rm RMS}=\omega>0$, the RMS iteration~\eqref{eq:rms} turns into the modified SOR-like iteration
      \cite[Equation (11)]{liwu2020}.

  \item[(iv)] If $B=I, M_{\rm RMS}=A, \tau_{\rm RMS}=\tau_{\rm FPI}>0$,  the RMS iteration \eqref{eq:rms} is reduced to the fixed point iteration (FPI) \cite{ke2020}
\begin{equation}\label{eq:fpi}
\begin{cases}
		x^{(k+1)}= A^{-1}(y^{(k)} + b),\\
        y^{(k+1)}= \left(1-\tau_{\rm FPI} \right)y^{(k)} + \tau_{\rm FPI}|x^{(k+1)}|,
\end{cases}
\end{equation}
which is a special case of the new SOR-like (NSOR) iteration \cite{doss2020}
\begin{equation}\label{eq:nsor}
\begin{cases}
x^{(k+1)}= x^{(k)} + \omega_{\rm NSOR} A^{-1}(b-Ax^{(k)} + y^{(k)} ),\\
y^{(k+1)}= y^{(k)} + \dfrac{\omega_{\rm NSOR} }{\sigma_{\rm NSOR}}( b + |x^{(k+1)}| - Ax^{(k+1)}),
\end{cases}
\end{equation}
where $\omega_{\rm NSOR},\sigma_{\rm NSOR}>0$. Indeed, if $\omega_{\rm NSOR}=1$ and $\sigma_{\rm NSOR} = \dfrac{1}{\tau_{\rm FPI}}$, the NSOR iteration \eqref{eq:nsor} reduces to \eqref{eq:fpi}. In addition, the RMS iteration \eqref{eq:rms} is equivalent to the NSOR iteration \eqref{eq:nsor} if $B = I$, $M_{\rm RMS}=A$, $\tau_{\rm RMS} = \frac{1}{\sigma_{\rm NSOR}}$ and $\omega_{\rm NSOR}=1$. Under these conditions, both of them can be seem as the FPI iteration.
\end{enumerate}

Before going ahead, we should point out that the FPI \eqref{eq:fpi} is also a special case of the modified FPI (MFPI) \cite{yuch2022}
\begin{equation}\label{eq:mfpi}
\begin{cases}
x^{(k+1)}= A^{-1}(Q_{\rm MFPI} y^{(k)} + b ),\\
y^{(k+1)}= (1-\tau_{\rm MFPI} )y^{(k)} + \tau_{\rm MFPI} Q^{-1}_{\rm MFPI}|x^{(k+1)}|,
\end{cases}
\end{equation}
where $Q_{\rm MFPI}\in \mathbb{R}^{n\times n}$ is a nonsingular matrix and $\tau_{\rm MFPI}>0$. Generally, NSOR~\eqref{eq:nsor} and MFPI~\eqref{eq:mfpi} are different from each other.

Though the RMS iteration \eqref{eq:rms} contains many existing iterations as its special cases, after further investigation, it cannot include the general NSOR iteration \eqref{eq:nsor} and the general MFPI~\eqref{eq:mfpi} (see the next section for more details). This motivates us to generalize the RMS iteration \eqref{eq:rms} such that, at least, the generalization includes the general NSOR iteration \eqref{eq:nsor} and the general MFPI~\eqref{eq:mfpi} as its special cases. To this end, the momentum acceleration technique \cite{zhang2013,rume1986,bk2004,qian1999} enters our vision. Concretely, by incorporating the momentum acceleration technique and a new matrix splitting into the RMS iteration \eqref{eq:rms}, a generalization of the RMS (GRMS) iteration is developed for solving GAVEs~\eqref{eq:gave}.

The rest of this paper is organized as follows. In Section \ref{sec:GRMS}, we establish the GRMS iterative method for solving GAVEs \eqref{eq:gave}. The convergence theories of the GRMS method are presented in Section \ref{sec:Convergence}.  In Section \ref{sec:compar}, comparison theorems between GRMS method and some existing methods are proposed. In Section \ref{sec:Numerical}, the feasibility and efficiency of the GRMS method are verified by numerical experiments. Finally, the conclusion of this paper is given in Section  \ref{sec:Conclusions}.

\textbf{Notation.}
$\mathbb{N}^+$ is the set of all positive integer. $\mathbb{R}^{n\times n}$ is the set of all $n \times n$ real matrices, $\mathbb{R}^{n}= \mathbb{R}^{n\times 1}$ and $\mathbb{R}^+$ represents the set of the positive real number. $|U|\in\mathbb{R}^{n\times n}$ denote the componentwise absolute value of matrix $U$. $I$ denotes the identity matrix with suitable dimension. For $U\in\mathbb{R}^{n\times n}$, $\rho(U)$ denotes the spectral radius of $U$, $\Vert U\Vert$ denotes the $2$-norm of $U$ which is defined by the formula $\Vert U\Vert=\max\{\Vert Ux\Vert:x\in\mathbb{R}^n,\Vert x\Vert=1\}$, where $\Vert x\Vert$ is the $2$-norm  of the vector $x$. For any matrices $U=(u_{ij})$ and $V=(v_{ij})\in\mathbb{R}^{n\times n}$, $U\leq V$ ($U< V$) means that $u_{ij}\leq v_{ij}$ ($u_{ij}< v_{ij}$) for $i,j=1,2,\ldots,n$.

\section{The GRMS iterative method }\label{sec:GRMS}
Let $|x|=Qy$ with $Q\in\mathbb{R}^{n\times n}$.  Then GAVEs \eqref{eq:gave} can be transformed into the following two-by-two block nonlinear equations
\begin{equation}\label{eq:two}
\begin{cases}
	Ax - BQy = b,\\
    Qy - |x| = 0.
\end{cases}
\end{equation}
By splitting $A$ and $Q$ as $ A= M-N$ and $Q = Q_1-Q_2$ with $M$ and $Q_1\in\mathbb{R}^{n\times n}$ being invertible. Then the system \eqref{eq:two} can be rewritten as
\begin{equation}\label{eq:etwo}
	\begin{cases}
    Mx = Nx + BQy + b,\\
    Q_1y= Q_1y -  \theta Q_1y + \theta Q_2y + \theta|x| +  Hx -  Hx ,
	\end{cases}
\end{equation}
where $\theta>0$ is a scaled parameter and $H\in\mathbb{R}^{n\times n}$ is a given matrix. Based on \eqref{eq:etwo}, the following Algorithm~\ref{alg:gsor} (which is called as the GRMS iteration method) can be developed for solving GAVEs \eqref{eq:gave}.

\begin{algorithm}
\caption{}\label{alg:gsor}
Assume that $A,B\in \mathbb{R}^{n\times n}$ and $Q\in \mathbb{R}^{n\times n}$ are given. Let $A=M-N$ and $Q = Q_1-Q_2$ with $M$ and $Q_1$ being invertible. Given initial vectors $x^{(0)}\in \mathbb{R}^{n}$ and $y^{(0)}\in \mathbb{R}^{n}$, for $k=0,1,2,\ldots$ until the iteration sequence~$\{(x^{(k)},y^{(k)})\}^\infty_{k=0} $ is convergent, compute
	\begin{equation}\label{eq:grms}
		\begin{cases}
            x^{(k+1)}=M^{-1}Nx^{(k)} + M^{-1}BQy^{(k)} + M^{-1}b,\\
			y^{(k+1)}=(1 - \theta )y^{(k)} + \theta {Q_1}^{-1}Q_2y^{(k)} +  \theta{Q_1}^{-1}|x^{(k+1)}| + {Q_1}^{-1}Hx^{(k+1)}-{Q_1}^{-1}Hx^{(k)},
		\end{cases}
	\end{equation}
in which $\theta>0$ is a scaled factor and $H\in \mathbb{R}^{n\times n}$ is a given matrix.
\end{algorithm}

Comparing with the RMS iteration \eqref{eq:rms}, the GRMS iteration \eqref{eq:grms} involves a new matrix $Q$ as well as its splitting and the momentum acceleration term ${Q_1}^{-1}H(x^{(k+1)}-x^{(k)})$ whenever $H\neq 0$, which make the GRMS iteration \eqref{eq:grms} more general than the RMS iteration~\eqref{eq:rms}. Indeed, if $M = M_{\rm RMS}, Q=I$, $Q_1 = \frac{\theta}{\tau_{\rm RMS}}I$, and $H=0$, the GRMS iteration \eqref{eq:grms} turns into the RMS iteration \eqref{eq:rms}. This also means that the
 GRMS iteration~\eqref{eq:grms} includes the existing iterative schemes contained in the RMS iteration~\eqref{eq:rms} as its special cases. Moreover, the GRMS iteration~\eqref{eq:grms} contains some iterative schemes which are not generally contained by the RMS iteration \eqref{eq:rms}, examples are as follows.

If $B = Q= I$, $M = \frac{1}{\alpha_{\rm MSOR}\omega_{\rm MSOR}}A$, $Q_1 = \frac{\theta(2-\omega_{\rm MSOR})}{2\omega_{\rm MSOR}}I$ and $H = \frac{\theta(1-\alpha_{\rm MSOR}\omega_{\rm MSOR})}{\alpha_{\rm MSOR}\omega_{\rm MSOR}}A$ ($\alpha_{\rm MSOR},\omega_{\rm MSOR} >0$ and $\omega_{\rm MSOR} \neq 2$), the GRMS iteration~\eqref{eq:grms} reduces to
\begin{equation}\label{eq:socmsor}
\begin{cases}
x^{(k+1)}=(1-\alpha_{\rm MSOR}\omega_{\rm MSOR})x^{(k)} + \alpha_{\rm MSOR}\omega_{\rm MSOR} A^{-1} (y^{(k)} + b),\\
y^{(k+1)}=y^{(k)} + \frac{2\omega_{\rm MSOR}}{2-\omega_{\rm MSOR}}(-A x^{(k+1)} + |x^{(k+1)}| + b),
\end{cases}
\end{equation}
which is the modified SOR-like (MSOR) iteration proposed in \cite{huli2022} in the context of AVEs~\eqref{eq:ave} associated with second order cones. Particularly, if $\alpha_{\rm MSOR} = \frac{2\sigma_{\rm NSOR} + \omega_{\rm NSOR}}{2}$ and $\omega_{\rm MSOR} = \frac{2\omega_{\rm NSOR}}{2\sigma_{\rm NSOR} + \omega_{\rm NSOR}}$, the iteration \eqref{eq:socmsor} reduces to the NSOR iteration \eqref{eq:nsor}. However, the RMS iteration~\eqref{eq:rms} can only contain the NSOR iteration \eqref{eq:nsor} with $\omega_{\rm NSOR} = 1$.

If $B=I$, $M = \frac{1}{\alpha_{\rm MGSOR}}A$, $Q=Q_{\rm MGSOR}$, $Q_1 = \frac{\theta}{\beta_{\rm MGSOR}} Q_{\rm MGSOR}$ and $H=0$ , then the GRMS iteration \eqref{eq:grms} turns into the modified generalized SOR-like (MGSOR) iteration~\cite{zhzl2023}
\begin{equation}\label{eq:mgsor}
\begin{cases}
x^{(k+1)}=(1-\alpha_{\rm MGSOR})x^{(k)} + \alpha_{\rm MGSOR} A^{-1} (Q_{\rm MGSOR} y^{(k)} + b),\\
y^{(k+1)}=(1-\beta_{\rm MGSOR})y^{(k)} + \beta_{\rm MGSOR}Q_{\rm MGSOR}^{-1}|x^{(k+1)}|,
\end{cases}
\end{equation}
where $\alpha_{\rm MGSOR}$ and $\beta_{\rm MGSOR}$ are positive constants and $Q_{\rm MGSOR}$ is a nonsingular matrix. Specially, if $\alpha_{\rm MGSOR} = \beta_{\rm MGSOR} =\omega_{\rm SOR}$ and $Q_{\rm MGSOR} = I$, the MGSOR iteration~\eqref{eq:mgsor} reduces to the SOR-like iteration \eqref{it:sor}. If $\alpha_{\rm MGSOR} = 1$, $\beta_{\rm MGSOR} = \tau_{\rm MFPI}$ and $Q_{\rm MGSOR} = Q_{\rm MFPI}$, the MGSOR iteration \eqref{eq:mgsor} reduces to the MFPI iteration~\eqref{eq:mfpi} \cite{zhzl2023}. However, the RMS iteration~\eqref{eq:rms} cannot contain the MGSOR iteration \eqref{eq:mgsor} with $Q_{\rm MGSOR}\neq I$ and the MFPI iteration~\eqref{eq:mfpi} with $Q_{\rm MFPI}\neq I$.

If $M = M_{\rm MAMS}+\Omega_{\rm MAMS}$ with $A = M_{\rm MAMS} - N_{\rm MAMS}$ and $\Omega_{\rm MAMS}$ being a positive diagonal matrix, $B= Q =I$, $Q_1=\theta I$ and $H=\theta\beta_{\rm MAMS} I$ ($\beta_{\rm MAMS}$ is a real number called momentum factor), the GRMS iteration \eqref{eq:grms} becomes the momentum acceleration-based matrix splitting (MAMS) iteration \cite{zzll2023}
\begin{equation}\label{it:mams1}
x^{(k+2)}= (M_{\rm MAMS}+\Omega_{\rm MAMS})^{-1}\left[(N_{{\rm MAMS}}+\Omega_{\rm MAMS})x^{(k+1)} + |x^{(k+1)}| + \beta_{\rm MAMS} (x^{(k+1)}-x^{(k)}) + b \right]
\end{equation}
with $x^{(1)} = (M_{\rm MAMS}+\Omega_{\rm MAMS})^{-1} [(N_{{\rm MAMS}}+\Omega_{\rm MAMS})x^{(0)} + y^{(0)} + b]$, where $x^{(0)}$ is used for both iterations and $y^{(0)}$ is used in the GRMS iteration. In other words, with suitable initial vectors, the GRMS iteration \eqref{eq:grms} is the same with the MAMS iteration \eqref{it:mams1} for solving AVEs~\eqref{eq:ave}. However, it seems impossible to establish such a relationship between the RMS iteration \eqref{eq:rms} and the  MAMS iteration \eqref{it:mams1} with $\beta_{\rm MAMS}\neq 0$.

In conclusion, the GRMS iteration \eqref{eq:grms} is more general than the RMS
iteration~\eqref{eq:rms}. Before ending this section, we will summarize more specific versions of the GRMS iteration \eqref{eq:grms} for solving GAVEs~\eqref{eq:gave}.

\begin{enumerate}
\item [(i)] If $M=\dfrac{1}{2}\left(\alpha_{\rm FPI\text{-}SS}I+ A\right)$ with $\alpha_{\rm FPI\text{-}SS}>0$, $Q=I, Q_1=\frac{\theta}{\omega_{\rm FPI\text{-}SS}}I,~H=0$, then the GRMS iteration~\eqref{eq:grms} turns into the shift-splitting fixed point iteration (FPI-SS) \cite{lild2022}
    \begin{equation}\label{eq:fpiss}
\begin{cases}
x^{(k+1)}=(\alpha_{\rm FPI\text{-}SS}I+ A)^{-1}(\alpha_{\rm FPI\text{-}SS}I- A)x^{(k)} + 2 (\alpha_{\rm FPI\text{-}SS}I+ A)^{-1} (B y^{(k)} + b),\\
y^{(k+1)}=(1-\omega_{\rm FPI\text{-}SS})y^{(k)} + \omega_{\rm FPI\text{-}SS}|x^{(k+1)}|,
\end{cases}
\end{equation}
where $\alpha_{\rm FPI\text{-}SS}$ is a positive constant such that $\alpha_{\rm FPI\text{-}SS}I+ A$ is nonsingular and $\omega_{\rm FPI\text{-}SS}$ is a positive constant. Particularly, when $M=A$, $Q=I, Q_1=\frac{\theta}{\omega_{\rm FPI}}I$ and $H=0$, the GRMS iteration~\eqref{eq:grms} turns into the fixed point iteration (FPI) for solving GAVEs~\eqref{eq:gave} \cite{lild2022}
\begin{equation}\label{eq:fpi4gave}
\begin{cases}
x^{(k+1)}=A^{-1}(B y^{(k)} + b),\\
y^{(k+1)}=(1-\omega_{\rm FPI})y^{(k)} + \omega_{\rm FPI}|x^{(k+1)}|,
\end{cases}
\end{equation}
 where $\omega_{\rm FPI}$ is a positive constant. Indeed, \eqref{eq:fpi4gave} is a generalization of \eqref{eq:fpi}.

\item [(ii)] If $M=\frac{1}{\alpha_{\rm MGSOR}}A$, $Q=Q_{\rm MGSOR}, Q_1=\frac{\theta}{\beta_{\rm MGSOR}}Q_{\rm MGSOR}$ and $H=0$,
     then the GRMS iteration~\eqref{eq:grms} becomes the MGSOR iteration
   \begin{equation}\label{eq:mgsor4gave}
		\begin{cases}
		x^{(k+1)}=(1-\alpha_{\rm MGSOR})x^{(k)}+\alpha_{\rm MGSOR} A^{-1}\left(BQ_{\rm MGSOR}y^{(k)} + b\right),\\
		y^{(k+1)}=(1-\beta_{\rm MGSOR})y^{(k)} + \beta_{\rm MGSOR} Q^{-1}_{\rm MGSOR} |x^{(k+1)}|
		\end{cases}
  \end{equation}
  for solving GAVEs~\eqref{eq:gave} with $\alpha_{\rm MGSOR}>0$ and $\beta_{\rm MGSOR}>0$, which is an extension of \eqref{eq:mgsor}. Specially, similar to the case for AVEs~\eqref{eq:ave}, if $\alpha_{\rm MGSOR} = \beta_{\rm MGSOR} =\omega_{\rm SOR}>0$ and $Q_{\rm MGSOR} = I$, the MGSOR iteration \eqref{eq:mgsor4gave} becomes the SOR-like iteration
    \begin{equation}\label{it:sor4gave}
		\begin{cases}
		    x^{(k+1)}=(1-\omega_{\rm SOR})x^{(k)} + \omega_{\rm SOR} A^{-1}\left( B y^{(k)} + b\right),\\
			y^{(k+1)}=(1-\omega_{\rm SOR})y^{(k)} + \omega_{\rm SOR} |x^{(k+1)}|
		\end{cases}
	\end{equation}
for solving GAVEs~\eqref{eq:gave}, which is a generalization of \eqref{it:sor}. If $\alpha_{\rm MGSOR} = 1$, $\beta_{\rm MGSOR} = \tau_{\rm MFPI}>0$ and $Q_{\rm MGSOR} = Q_{\rm MFPI}$, the MGSOR iteration \eqref{eq:mgsor4gave} turns into the following MFPI
\begin{equation}\label{eq:mfpi4gave}
\begin{cases}
x^{(k+1)}= A^{-1}(BQ_{\rm MFPI} y^{(k)} + b ),\\
y^{(k+1)}= (1-\tau_{\rm MFPI} )y^{(k)} + \tau_{\rm MFPI} Q^{-1}_{\rm MFPI}|x^{(k+1)}|
\end{cases}
\end{equation}
for solving GAVEs~\eqref{eq:gave}, which is a generalization of \eqref{eq:mfpi}.

\item [(iii)] Let $B$ be nonsingular. If $M = \frac{1}{\alpha_{\rm MSOR}\omega_{\rm MSOR}}A$, $Q = I$, $Q_1 = \frac{\theta (2-\omega_{\rm MSOR})}{2\omega_{\rm MSOR}}B^{-1}$ and $H = \frac{\theta(1- \alpha_{\rm MSOR}\omega_{\rm MSOR})}{\alpha_{\rm MSOR}\omega_{\rm MSOR}}B^{-1}A$ ($\alpha_{\rm MSOR},\omega_{\rm MSOR} >0$ and $\omega_{\rm MSOR} \neq 2$), the GRMS iteration~\eqref{eq:grms} becomes the MSOR iteration
    \begin{equation}\label{eq:msor4gave}
\begin{cases}
x^{(k+1)}=(1-\alpha_{\rm MSOR}\omega_{\rm MSOR})x^{(k)} + \alpha_{\rm MSOR}\omega_{\rm MSOR} A^{-1} (B y^{(k)} + b),\\
y^{(k+1)}=y^{(k)} + \frac{2\omega_{\rm MSOR}}{2-\omega_{\rm MSOR}}(-A x^{(k+1)} + B|x^{(k+1)}| + b)
\end{cases}
\end{equation}
for solving GAVEs~\eqref{eq:gave}, which is a generalization of \eqref{eq:socmsor}. Particularly, if $\alpha_{\rm MSOR} = \frac{2\sigma_{\rm NSOR} + \omega_{\rm NSOR}}{2}$ and $\omega_{\rm MSOR} = \frac{2\omega_{\rm NSOR}}{2\sigma_{\rm NSOR} + \omega_{\rm NSOR}}$,  the MSOR iteration \eqref{eq:msor4gave} reduces to the NSOR iteration
\begin{equation}\label{eq:nsor4gave}
\begin{cases}
x^{(k+1)}= x^{(k)} + \omega_{\rm NSOR} A^{-1}(b-Ax^{(k)} + By^{(k)} ),\\
y^{(k+1)}= y^{(k)} + \dfrac{\omega_{\rm NSOR} }{\sigma_{\rm NSOR}}( b + B|x^{(k+1)}| - Ax^{(k+1)})
\end{cases}
\end{equation}
for solving GAVEs~\eqref{eq:gave} with $\omega_{\rm NSOR}>0$ and $\sigma_{\rm NSOR}>0$, which is a generalization of \eqref{eq:nsor}.

\item [(iv)] Let $B$ be nonsingular. If $\Omega_{\rm MAMS}$ is a positive diagonal matrix such that $M=M_{\rm MAMS}+\Omega_{\rm MAMS}$ is nonsingular, where $A=M_{\rm MAMS}-N_{\rm MAMS}$. Furthermore, let $Q=I$, $Q_1=\theta I$, and $H=\theta\beta_{\rm MAMS}B^{-1}$ with $\beta_{\rm MAMS}$. Then the GRMS iteration~\eqref{eq:grms} turns into the following MAMS iteration \cite{zzll2023}
 \begin{equation}\label{eq:mams4gave}
  x^{(k+2)}= (M_{\rm MAMS}+\Omega_{\rm MAMS})^{-1}\left[\big(N_{\rm MAMS}+\Omega_{\rm MAMS}\big)x^{(k+1)} + B|x^{(k+1)}| + \beta_{\rm MAMS} \big(x^{(k+1)}-x^{(k)}\big) + b \right]
\end{equation}
with $x^{(1)} = (M_{\rm MAMS}+\Omega_{\rm MAMS})^{-1} [(N_{{\rm MAMS}}+\Omega_{\rm MAMS})x^{(0)} + By^{(0)} + b]$, where $x^{(0)}$ is used for both iterations and $y^{(0)}$ is used for the GRMS iteration.

\end{enumerate}

\section{Convergence analysis}\label{sec:Convergence}
In this section, we will analyse the convergence of the GRMS iteration \eqref{eq:grms}. Indeed, we have the following convergence theorem.

\begin{theorem}\label{thm:cov}
Assume that $A,B,H\in \mathbb{R}^{n\times n}$, and $Q\in \mathbb{R}^{n\times n}$ is nonsingular. Let $A = M - N$ and $Q = Q_1 - Q_2$ with $M$ and $Q_1$ being nonsingular. Denote
\begin{equation}\label{eq:nota}
a=\Vert M^{-1}N\Vert, c=\Vert M^{-1}BQ\Vert, d=\Vert Q_1^{-1}Q_2\Vert, \alpha=\Vert Q_1^{-1}\Vert, \beta=\Vert Q_1^{-1}H\Vert.
\end{equation}
If
\begin{equation}\label{eq:con1}
\big|a|1-\theta|+\theta ad - c\beta\big|<1
\end{equation}
and
\begin{equation}\label{eq:con2}
c(\theta\alpha + 2\beta)<(1 - a)(1 - |1-\theta|-\theta d),
\end{equation}
then GAVEs \eqref{eq:gave} has a unique solution $x^*$ for any $b\in \mathbb{R}^n$ and
the sequence $\{(x^{(k)},y^{(k)})\}_{k=0}^{\infty}$ generated by the GRMS iteration \eqref{eq:grms}converges to $(x^*, y^*=Q^{-1}|x^*|)$.
\end{theorem}

\begin{proof}
Let
$$e^{(k+1)}_x=x^{(k+1)} - x^{(k)} \quad \text{and}\quad e^{(k+1)}_y = y^{(k+1)} - y^{(k)},$$
and then from \eqref{eq:grms}, for $k \geq 1$, we have
\begin{equation}\label{eq:er}
\begin{cases}
	e^{(k+1)}_x= M^{-1}Ne^{(k)}_x + M^{-1}BQe^{(k)}_y,\\		
    e^{(k+1)}_y= (1 - \theta)e^{(k)}_y + \theta Q_1^{-1}Q_2e^{(k)}_y + \theta Q_1^{-1}(|x^{(k+1)}| - |x^{(k)}|) + Q_1^{-1}He^{(k+1)}_x - Q_1^{-1}He^{(k)}_x.
\end{cases}
\end{equation}
It follows from \eqref{eq:nota} and  \eqref{eq:er} that
\begin{equation}\label{neq:errors}
\begin{cases}
	\Vert e^{(k+1)}_x\Vert\leq a\Vert e^{(k)}_x\Vert + c\Vert e^{(k)}_y\Vert,\\	
    \Vert e^{(k+1)}_y\Vert\leq \big(|1-\theta| + \theta d\big)\Vert e^{(k)}_y\Vert + \big(\theta\alpha + \beta\big)\Vert e^{(k+1)}_x\Vert + \beta\Vert e^{(k)}_x\Vert,
\end{cases}
\end{equation}
in which the inequality $\||x^{(k+1)}|-|x^{(k)}|\|\le \|x^{(k+1)}-x^{(k)}\|$ is used. It follows from \eqref{neq:errors} that
\begin{equation}\label{eq:errors}
		\begin{bmatrix}
			        1       & 0 \\
			-(\theta\alpha + \beta) & 1
		\end{bmatrix}
		\begin{bmatrix}
			\Vert e_x^{(k+1)}\Vert\\
			\Vert e_y^{(k+1)}\Vert
		\end{bmatrix}
		\leq
		\begin{bmatrix}
			a & c \\
			\beta & |1 - \theta| + \theta d
		\end{bmatrix}
		\begin{bmatrix}
			\Vert e_x^{(k)}\Vert\\
			\Vert e_y^{(k)}\Vert
\end{bmatrix}.			
\end{equation}
Left multiplying both sides of the inequality \eqref{eq:errors} by the nonnegative matrix
$\begin{bmatrix}
			1       & 0 \\
			\theta\alpha + \beta & 1
\end{bmatrix}$,
we have
 \begin{equation}\label{neq:errnorm}
 	\begin{bmatrix}
 		\Vert e_x^{(k+1)}\Vert\\
 		\Vert e_y^{(k+1)}\Vert
 	\end{bmatrix}
 	\leq T
 	\begin{bmatrix}
 		\Vert e_x^{(k)}\Vert\\
 		\Vert e_y^{(k)}\Vert
 	\end{bmatrix},
 \end{equation}
where
\begin{equation}\label{eq:t}
T=
 \begin{bmatrix}
 		a & c\\
    a(\theta\alpha + \beta) + \beta & c(\theta\alpha + \beta) + \theta d + |1-\theta|
 	\end{bmatrix}.
 \end{equation}
From \eqref{neq:errnorm}, we can deduce that
\begin{equation}\label{eq:errorsnorm}
	\begin{bmatrix}
 		\Vert e_x^{(k+1)}\Vert\\
 		\Vert e_y^{(k+1)}\Vert
 	\end{bmatrix}
	\leq  T
    \begin{bmatrix}
 		\Vert e_x^{(k)}\Vert\\
 		\Vert e_y^{(k)}\Vert
 	\end{bmatrix}
         \leq\dots\leq
	      T^{k}
	\begin{bmatrix}
 		\Vert e_x^{(1)}\Vert\\
 		\Vert e_y^{(1)}\Vert
 	\end{bmatrix}.
\end{equation}
Then for $k\geq1$ and for any $m\in\mathbb{N^+}$, if $\rho(T)<1$, it follows from \eqref{eq:errorsnorm} that
\begin{align}\nonumber
\begin{bmatrix}
	\Vert x^{(k+m)} - x^{(k)} \Vert\\
	\Vert y^{(k+m)} - y^{(k)} \Vert
	\end{bmatrix}
	&=
		\begin{bmatrix}
			\Big\Vert \sum\limits_{j=0}^{m-1} e^{(k+j+1)}_x \Big\Vert\\
			\Big\Vert\sum\limits_{j=0}^{m-1} e_y^{(k+j+1)}  \Big\Vert
		\end{bmatrix}
		 \leq
		\begin{bmatrix}
			\sum\limits_{j=0}^{m-1}\Vert e^{(k+j+1)}_x \Vert\\
			\sum\limits_{j=0}^{m-1}\Vert e_y^{(k+j+1)} \Vert
		\end{bmatrix} = \sum\limits_{j=0}^{m-1}\begin{bmatrix}
			\Vert e^{(k+j+1)}_x \Vert\\
			\Vert e_y^{(k+j+1)} \Vert
		\end{bmatrix} \\\nonumber
		&\le
		\left(\sum_{j=0}^{m-1}T^j\right)\left( \begin{bmatrix}
			\Vert e^{(k+1)}_x \Vert\\
			\Vert e_y^{(k+1)} \Vert
		\end{bmatrix}\right)\le \left(\sum_{j=0}^{m-1}T^j\right)\left( T^k
		\begin{bmatrix}
		    \Vert e_x^{(1)} \Vert\\
		    \Vert e_y^{(1)}\Vert
		\end{bmatrix}\right)\\\label{ie:cx}
&\le \left(\sum_{j=0}^{\infty}T^j\right)\left( T^k
		\begin{bmatrix}
		    \Vert e_x^{(1)} \Vert\\
		    \Vert e_y^{(1)}\Vert
		\end{bmatrix}\right)=(I-T)^{-1}T^k
		\begin{bmatrix}
		    \Vert e_x^{(1)} \Vert\\
		    \Vert e_y^{(1)}\Vert
		\end{bmatrix}.
	\end{align}
In addition, if $\rho(T)<1$, and then $\lim\limits_{k\rightarrow\infty}T^k=0$. Hence, for  $\forall\varepsilon>0$, there exist $K\in\mathbb{N}^+$ such that for $k>K$ we obtain
 \begin{equation}\label{3.7}
 |T^k|<\begin{bmatrix}
			\varepsilon & \varepsilon\\
			\varepsilon & \varepsilon
		\end{bmatrix}.
 \end{equation}
Moreover, for $(I-T)^{-1}$ and $\begin{bmatrix}
			\Vert e_x^{(1)} \Vert\\
			\Vert e_y^{(1)} \Vert
\end{bmatrix}$, there is an $s>0$ such that
 \begin{equation}\label{3.8}
 |(I-T)^{-1}|<\begin{bmatrix}
			s & s\\
			s & s
		\end{bmatrix}\quad \text{and}\quad
\left|\begin{bmatrix}
			\Vert e_x^{(1)} \Vert\\
			\Vert e_y^{(1)} \Vert
\end{bmatrix}\right|
 <\begin{bmatrix}
			s \\
			s
\end{bmatrix}.
\end{equation}
According to \eqref{ie:cx}, \eqref{3.7} and \eqref{3.8}, if $\rho(T)<1$, and then for all $\hat{\varepsilon}=4s^2 \varepsilon>0$,
$\exists K>0$~(as mentioned above) such that for all $k>K$ and any $m\in\mathbb{N}^+$ we have
 \begin{equation}\label{eq:cr}
\begin{bmatrix}
			\Vert x^{(k+m)} - x^{(k)} \Vert\\
			\Vert y^{(k+m)} - y^{(k)} \Vert
		\end{bmatrix}
=\left|\begin{bmatrix}
			\Vert x^{(k+m)} - x^{(k)} \Vert\\
			\Vert y^{(k+m)} - y^{(k)} \Vert
		\end{bmatrix}\right|
		\leq
		\left|(I-T)^{-1}\right| \left|T^k\right|\left|
		\begin{bmatrix}
			\Vert e_x^{(1)} \Vert\\
			\Vert e_y^{(1)} \Vert
		\end{bmatrix}\right|
		<
        \begin{bmatrix}
            \hat{\varepsilon} \\
            \hat{\varepsilon} \\
         \end{bmatrix},
	\end{equation}
from which we can conclude that both $\{x^{(k)}\}^\infty_{k=0}$ and $\{y^{(k)}\}^\infty_{k=0}$ are Cauchy sequences. Since the Cauchy sequences $\{x^{(k)}\}^\infty_{k=0} \subseteq \mathbb{R}^n$ and $\{y^{(k)}\}^\infty_{k=0}\subseteq \mathbb{R}^n$, they are convergent. Hence, we can let $\mathop{\lim}\limits_{k\rightarrow\infty}x^{(k)} = x^*$ and $\mathop{\lim}\limits_{k\rightarrow\infty}y^{(k)} = y^*$. Then it follows from \eqref{eq:grms} that
\begin{equation*}
		\begin{cases}
		x^* = M^{-1}( Nx^* + BQy^* + b ),\\
        y^*= Q_1^{-1}\left[( 1 - \theta )Q_1y^* + \theta Q_2y^* + \theta |x^*| + Hx^* - Hx^*\right],
		\end{cases}
	\end{equation*}
which implies that
\begin{equation*}
		\begin{cases}
     Ax^* - BQy^* - b = 0,\\
     Qy^*= |x^*|,
		\end{cases}
	\end{equation*}
that is,
\begin{equation}\label{eq:lgnms}
		\begin{cases}
     Ax^* - B|x^*| - b = 0,\\
     y^*= Q^{-1}|x^*|.
		\end{cases}
	\end{equation}
The first equation of \eqref{eq:lgnms} means that $x^*$ is a solution of GAVEs~\eqref{eq:gave}.

In the following, we will prove that $\rho(T)<1$ if \eqref{eq:con1} and \eqref{eq:con2} hold. Suppose that  $\lambda$ is an eigenvalue of $T$, from \eqref{eq:t}, we have
	\begin{equation*}
		\det( \lambda I - T ) = \lambda^2 - \left[ c ( \theta\alpha + \beta ) + a + |1 - \theta| + \theta d \right]\lambda + a|1 - \theta| + \theta ad - c\beta=0.
	\end{equation*}
   Using \cite[Lemma 2.1]{young1971}, $\rho(T) < 1$ provided that
	\begin{align*}
		\left|a|1 - \theta| + \theta ad - c\beta\right|&<1,\\
        c( \theta\alpha + 2\beta )&<( 1 - a )( 1 - |1 - \theta| - \theta d ).
	\end{align*}
In other words, \eqref{eq:con1} and \eqref{eq:con2} implies $\rho(T) < 1$.

The final step is to prove the uniqueness of the solution to GAVEs \eqref{eq:gave}. Assume that there is another solution $\bar{x}^*$ to GAVEs \eqref{eq:gave}, and then we can uniquely determine $\bar{y}^* = Q^{-1}|\bar{x}^*|$. Using \eqref{eq:etwo},  we have
    \begin{equation}\label{eq:gnmssol}
     \begin{aligned}
   	   x^* - \bar{x}^*&= M^{-1}N ( x^* - \bar{x}^* ) + M^{-1}BQ( y^* - \bar{y}^* ),\\
       y^* - \bar{y}^*&= ( 1 - \theta )( y^* - \bar{y}^* ) + \theta Q_1^{-1}Q_2( y^* - \bar{y}^* )\\
         & \qquad + \theta Q_1^{-1}( |x^*| - |\bar{x}^*| ) + Q_1^{-1}H( x^* - \bar{x}^* ) - Q_1^{-1}H( x^* - \bar{x}^* ).
	 \end{aligned}
   \end{equation}
It follows from \eqref{eq:gnmssol} and \eqref{eq:nota}  that
    \begin{align}
    &\Vert x^* - \bar{x}^*\Vert\leq a \Vert x^* - \bar{x}^*\Vert + c\Vert y^* - \bar{y}^*\Vert,\label{eq:3.8a}\\
    &\Vert y^* - \bar{y}^*\Vert\leq ( |1 - \theta| + \theta d )\Vert y^* - \bar{y}^*\Vert + ( \theta\alpha + 2\beta )\Vert x^* - \bar{x}^*\Vert.\label{eq:3.8b}
    \end{align}
If \eqref{eq:con1} and \eqref{eq:con2} hold, we can conclude that $a<1~\text{and}~|1-\theta| + \theta d<1$. In fact, the condition \eqref{eq:con2} implies that either
$$a<1\quad\text{and}\quad |1-\theta| + \theta d<1$$
or
$$a>1\quad \text{and}\quad |1-\theta| + \theta d>1$$
holds. If $a>1$  and $|1-\theta| + \theta d>1$, we have \begin{equation}\label{ie:ie1}
 a(|1-\theta| + \theta d)>1 \quad \text{and}\quad a + |1-\theta| + \theta d>2.
\end{equation}
In addition, it follows from \eqref{eq:con2} that
\begin{equation}\label{ie:ie2}
   \theta\alpha c < -2c\beta + 1 -( a + |1 - \theta| + \theta d )+ a(|1 - \theta| + \theta d ).
\end{equation}
On the other hand, from the condition \eqref{eq:con1},  we obtain
\begin{equation}\label{ie:ie3}
   -2(a|1-\theta| + \theta ad + 1)<- 2c\beta<2(1-a|1-\theta| - \theta ad).
\end{equation}
Hence, it follows from \eqref{ie:ie1}, \eqref{ie:ie2} and \eqref{ie:ie3} that
\begin{align*}
0<\theta\alpha c < & ~2 - 2(a|1 - \theta| + \theta ad ) + 1 - ( a + |1 - \theta| + \theta d ) + a|1 - \theta| + \theta ad \\
     = & 3 - (a + |1-\theta| + \theta d) - a(|1 - \theta| + \theta d)<0,
\end{align*}
which is a contradiction. Hence, if \eqref{eq:con1} and \eqref{eq:con2} hold, we have $a<1$ and $|1-\theta| + \theta d<1$.

Then it follows from \eqref{eq:3.8b} that
   \begin{equation}\label{eq:x}
   	\Vert y^* - \bar{y}^*\Vert\leq \dfrac{\theta\alpha + 2\beta}{1 - |1 - \theta| - \theta d}\Vert x^* - \bar{x}^*\Vert.
   \end{equation}
According to \eqref{eq:x}, \eqref{eq:3.8a} and \eqref{eq:con2}, we obtain
\begin{eqnarray*}
\Vert x^* - \bar{x}^*\Vert
&\leq& a\Vert x^* - \bar{x}^*\Vert + \dfrac{c(\theta\alpha + 2\beta)}{1 - |1 - \theta| - \theta d}\Vert x^* -  \bar{x}^*\Vert\\
&<& a\Vert x^* - \bar{x}^*\Vert + ( 1 - a )\Vert x^* - \bar{x}^*\Vert\\
&=& \Vert x^* - \bar{x}^*\Vert,
\end{eqnarray*}
which is a contradiction.
\end{proof}

\begin{remark}{\rm
In Theorem~\ref{thm:cov}, we assume that $Q$ is nonsingular so that $y^* = Q^{-1}|x^*|$ is obtained. If we are not interested in the limit of the auxiliary sequence $\{y^{(k)}\}_{k=0}^\infty$, the nonsingularity hypothesis of $Q$ is not required. In this case, the results of Theorem~\ref{thm:cov} except that of $\{y^{(k)}\}_{k=0}^\infty$ still hold. For the sake of clarity, we always remain the nonsingularity of $Q$ in the following.}
\end{remark}

Based on Theorem \ref{thm:cov}, we have the following Corollary~\ref{cor:con1}.

\begin{corollary}\label{cor:con1}
Assume that $A,B,H\in \mathbb{R}^{n\times n}$ and $Q\in \mathbb{R}^{n\times n}$ is nonsingular. Let $A = M - N$ and $Q = Q_1 - Q_2$ with $M$ and $Q_1$ being nonsingular. Let $a, c, d, \alpha$ and $\beta$ be defined as in \eqref{eq:nota}. If
    \begin{equation}\label{neq:con1}
    a<1,~c\beta<ad + 1,~c(2\beta+\alpha)<(1-a)(1-d),
    \end{equation}
and
	\begin{equation}\label{neq:con2}
		\dfrac{2c\beta}{(1-a)(1-d)-c\alpha}<\theta <\dfrac{2(1-a)-2c\beta}{(1-a)(1+d)+c\alpha},
	\end{equation}
then the sequence $\{(x^{(k)},~y^{(k)})\}^\infty_{k=0}$ generated by \eqref{eq:grms} converges to $(x^*,~Q^{-1}|x^*|)$ and $x^*$ is the unique solution of GAVEs \eqref{eq:gave}.
\end{corollary}

\begin{proof}
The corollary will be proved if we can show that \eqref{neq:con1} and \eqref{neq:con2} imply  \eqref{eq:con1} and \eqref{eq:con2}.  The proof is divided into two cases.

\begin{itemize}
  \item [(i)] We first consider $\dfrac{2c\beta}{(1-a)(1-d)-c\alpha}< \theta\leq 1$. It follows from \eqref{neq:con1} that
	$$d<1\quad \text{and} \quad c\beta<ad+1<a+1,$$
from which we have
    \begin{eqnarray}
    \dfrac{c\beta-(a-1)}{a(d-1)}<0<\theta,\label{ieq:1}\\
    \theta\leq1<\dfrac{c\beta-(a+1)}{a(d-1)}. \label{ieq:2}
    \end{eqnarray}
    From the inequality \eqref{ieq:1}, we have
	$$\theta ad - \theta a - c\beta +a<1,$$
    that is,
    \begin{equation}\label{eq:one}
    a(1-\theta)+ \theta ad - c\beta <1.
    \end{equation}
    From the inequality \eqref{ieq:2}, we have
	$$\theta ad - \theta a - c\beta +a>-1,$$
    that is,
    \begin{equation}\label{eq:second}
    a(1-\theta)+ \theta ad - c\beta >-1.
    \end{equation}
Note that $0<\theta\le 1$, combining \eqref{eq:one} and \eqref{eq:second}, we have
    $$\left| a|1-\theta| + \theta ad - c\beta \right|<1,$$
which is the inequality \eqref{eq:con1}.

On the other hand, it follows from \eqref{neq:con1} that $c\alpha<(1-a)(1-d)$ (otherwise, we get $2c \beta < 0$, a contradiction). Multiplying both sides of $\dfrac{2c\beta}{(1-a)(1-d)-c\alpha}< \theta$ by $(1-a)(1-d)-c\alpha$, we have
    \begin{equation}\label{neq:3.1-2}
    2c\beta<\theta(1-a)(1-d)-\theta c\alpha.
    \end{equation}
From \eqref{neq:3.1-2}, we infer that (remind that $0<\theta \le 1$)
	   \begin{equation*}
        c(\theta\alpha + 2\beta)<(1-a)(\theta + 1 - 1 - \theta d)=(1-a)(1-|1-\theta|-\theta d),
	   \end{equation*}
	    which is the inequality \eqref{eq:con2}.

  \item [(ii)] Now we turn to $1<\theta <\dfrac{2(1-a)-2c\beta}{(1-a)(1+d)+c\alpha}$. On the condition of \eqref{neq:con1}
      and~\eqref{neq:con2}, we have
  \begin{align}\label{ie:case1}
  \dfrac{c\beta+a-1}{a(d+1)}<1&<\theta,\\\label{ie:case2}
  \theta<\dfrac{2(1-a)-2c\beta}{(1-a)(1+d)+c\alpha}&<\dfrac{c\beta+a+1}{a(d+1)}.
  \end{align}
  From \eqref{ie:case1}, we have
        $$\theta ad + \theta a - a - c\beta>-1,$$
  that is,
        \begin{equation}\label{neq:3}
        a(\theta-1)+\theta ad - c\beta>-1.
        \end{equation}
        From \eqref{ie:case2}, we have
        $$\theta ad + \theta a - a - c\beta<1,$$
       that is,
        \begin{equation}\label{neq:4}
        a(\theta-1)+\theta ad - c\beta<1.
        \end{equation}
  Note that $\theta >1$, combining \eqref{neq:3} and \eqref{neq:4}, we have
        $$\left|a|1 - \theta|+\theta ad - c\beta \right|<1,$$
        which is the inequality \eqref{eq:con1}.

  On the other hand, if $a<1$, we have $(1-a)(1+d)+c\alpha>0$.
        According to the first inequality of \eqref{ie:case2}, we have
        $$\theta(1-a)(1+d)+\theta c \alpha<2(1-a)-2c\beta,$$
        that is,
        \begin{equation*}
        c(\theta\alpha + 2\beta)<(1-a)(2-\theta-\theta d)
        =(1-a)(1 - |1-\theta| - \theta d),
        \end{equation*}
      which is the inequality \eqref{eq:con2}.
\end{itemize}
The proof is completed by summarizing the results in (i) and (ii).
\end{proof}

In the following, we will apply the results of the Theorem \ref{thm:cov} to some existing methods.

For the RMS iteration method \cite{soso2023}, we have the following corollary.

\begin{corollary}\label{cor:rms}
Assume that $A,B\in \mathbb{R}^{n\times n}$ are given and $A = M_{\rm RMS} - N_{\rm RMS}$ with $M_{\rm RMS}$ being nonsingular, $\theta=\tau_{\rm RMS}>0$. If
\begin{equation}\label{ie:rms1}
|1-\tau_{\rm RMS}|\|M_{\rm RMS}^{-1}N_{\rm RMS}\|<1
\end{equation}
and
\begin{equation}\label{ie:rms2}
\tau_{\rm RMS} \|M_{\rm RMS}^{-1}B\| < (1-\|M_{\rm RMS}^{-1}N_{\rm RMS}\|)(1-|1-\tau_{\rm RMS}|),
\end{equation}
then GAVEs \eqref{eq:gave} has a unique solution $x^*$ for any $b\in \mathbb{R}^n$ and the sequence $\{x^{(k)}\}_{k=0}^{\infty}$ generated by the RMS iteration \eqref{eq:rms} converges to $x^*$.
\end{corollary}

\begin{remark}{\rm
The conditions \eqref{ie:rms1} and \eqref{ie:rms2} appear in \cite[Lemma~3]{soso2023}. However, the existence and uniqueness of the solution to GAVEs \eqref{eq:gave} is not considered there.}
\end{remark}

For the NMS iteration method \cite{zhwl2021}, we can obtain the following results.
\begin{corollary}\label{cor:nms}
Assume that $A,B,\Omega_{\rm NMS}\in \mathbb{R}^{n\times n}$ are given. Let $A = M_{\rm NMS} - N_{\rm NMS}$ with $M_{\rm NMS}+\Omega_{\rm NMS}$ being nonsingular. If
\begin{equation}\label{eq:nmsconv}
\Vert{(M_{\rm NMS}+\Omega_{\rm NMS})}^{-1}(N_{\rm NMS}+\Omega_{\rm NMS})\Vert+\Vert{(M_{\rm NMS}+\Omega_{\rm NMS})}^{-1}B\Vert <1,
\end{equation}
then GAVEs \eqref{eq:gave} has a unique solution $x^*$ for any $b\in \mathbb{R}^n$ and the sequence $\{x^{(k)}\}_{k=0}^{\infty}$ generated by the NMS iteration \cite{zhwl2021} converges to $x^*$.
\end{corollary}

\begin{remark}{\rm
In \cite[Corollary 4.2.]{zhwl2021}, the authors show that the sequence $\{x^{(k)}\}_{k=0}^{\infty}$ generated by the NMS iteration converges to a solution of GAVEs \eqref{eq:gave} if
\begin{equation}\label{eq:nms1}
\Vert(M_{\rm NMS} + \Omega_{\rm NMS})^{-1}\Vert(\Vert N_{\rm NMS} + \Omega_{\rm NMS}\Vert +\Vert B\Vert) < 1,
\end{equation}
which implies \eqref{eq:nmsconv} while the reverse is generally not true. Moreover, the existence and uniqueness of the solution to GAVEs \eqref{eq:gave} is not considered in \cite[Corollary 4.2.]{zhwl2021}.}
\end{remark}

\begin{corollary}\label{cor:mn}
Assume that $A,B\in \mathbb{R}^{n\times n}$ are given. Let $A+\Omega_{\rm MN}$ be nonsingular with $\Omega_{\rm MN}$ being a semi-definite matrix. If
\begin{equation}\label{eq:mnconv}
\Vert(A+\Omega_{\rm MN})^{-1}\Omega_{\rm MN}\Vert+\Vert(A+\Omega_{\rm MN})^{-1}B\Vert < 1,
 \end{equation}
then GAVEs \eqref{eq:gave} has a unique solution $x^*$ for any $b\in \mathbb{R}^n$ and the sequence $\{x^{(k)}\}_{k=0}^{\infty}$ generated by the MN iteration \cite{wacc2019} converges to $x^*$.
\end{corollary}

\begin{remark}\label{re:3.2}{\rm
In \cite[Theorem 3.1]{wacc2019}, the authors show that the sequence $\{x^{(k)}\}_{k=0}^{\infty}$ generated by the MN iteration converges to a solution of GAVEs~\eqref{eq:gave} if
\begin{equation}\label{eq:wc}
\|(A+\Omega_{\rm MN})^{-1}\|(\|\Omega_{\rm MN}\| + \|B\|)<1.
\end{equation}
It is obvious that \eqref{eq:wc} implies \eqref{eq:mnconv} while the reverse is generally not true. In addition, the existence and uniqueness of the solution to GAVEs \eqref{eq:gave} is not considered in \cite[Theorem 3.1]{wacc2019}.}
\end{remark}

\begin{corollary}\label{cor:pi}
Assume that $A,B\in \mathbb{R}^{n\times n}$ are given and $A$ is nonsingular. If $\Vert A^{-1}B\Vert< 1$,  then GAVEs \eqref{eq:gave} has a unique solution $x^*$ for any $b\in \mathbb{R}^n$ and the sequence $\{x^{(k)}\}_{k=0}^{\infty}$ generated by the Picard iteration \cite{rohf2014} converges to $x^*$.
\end{corollary}

\begin{remark}{\rm
 In \cite[Theorem 2]{rohf2014}, the authors show that the sequence $\{x^{(k)}\}_{k=0}^{\infty}$ generated by the Picard iteration converges to the unique solution $x^*$ of GAVEs \eqref{eq:gave} if  $\rho(|A^{-1}B|)<1$. From \cite[Theorem 5.6.9.]{prhc2006}, we have $\rho(|A^{-1}B|)\leq\Vert A^{-1}B\Vert$ if $A^{-1}B \geq 0$. However, if $A^{-1}B \ngeqslant0$, $\rho(|A^{-1}B|) < 1$  and $\Vert A^{-1}B\Vert < 1$ can be independent of each other. For instance, if $A=
			\begin{bmatrix}
				-10 & 6\\
				2 & -95
			\end{bmatrix}$
			and $B=
			\begin{bmatrix}
				7 & -9\\
				8 & 1
			\end{bmatrix}$, we have $A^{-1}B\approx
			\begin{bmatrix}
				-0.7601 & 0.9051\\
				-0.1002 & 0.0085
			\end{bmatrix}\ngeq 0$ with $\Vert A^{-1}B\Vert\approx1.1841>1$ and $\rho(\vert A^{-1}B\vert)\approx0.8659<1$. If $A=
			\begin{bmatrix}
				12 & -0.5\\
				-0.5 & 12
			\end{bmatrix}$
			 and $B=
			\begin{bmatrix}
				7 & -8\\
				-6 & -4
			\end{bmatrix}$, and then $A^{-1}B\approx
			\begin{bmatrix}
				0.5635 & -0.6817\\
				-0.4765 & -0.3617
			\end{bmatrix}\ngeq 0$ with $\Vert A^{-1}B\Vert\approx0.8851<1$ and $\rho(\vert A^{-1}B\vert)\approx1.0414>1$.}
\end{remark}

For the MGSOR iteration \cite{zhzl2023}, we have the following corollary.
\begin{corollary}\label{cor:exmgsor}
Assume that $A,B,Q_{\rm MGSOR}\in\mathbb{R}^{n\times n}$ are given and let $A$, $Q_{\rm MGSOR}$ be nonsingular matrices, $\theta=\beta_{\rm MGSOR}>0$ and $\alpha_{\rm MGSOR}>0$. If
\begin{subequations}\label{ie:mgsor1}
\begin{equation}
|1-\alpha_{\rm MGSOR}||1-\beta_{\rm MGSOR}|<1
\end{equation}
\text{and}
\begin{equation}
\alpha_{\rm MGSOR}\beta_{\rm MGSOR}\Vert A^{-1}BQ_{\rm MGSOR}\Vert\Vert Q_{\rm MGSOR}^{-1}\Vert<(1-|1-\alpha_{\rm MGSOR}|)(1-|1-\beta_{\rm MGSOR}|),
\end{equation}
\end{subequations}
then GAVEs \eqref{eq:gave} has a unique solution $x^*$ for any $b\in \mathbb{R}^n$ and the sequence $\{(x^{(k)}, y^{(k)})\}_{k=0}^{\infty}$ generated by the MGSOR iteration \eqref{eq:mgsor4gave} converges to $(x^*,y^*=Q_{\rm MGSOR}^{-1}|x^*|)$ for any given initial vectors $x^{(0)},y^{(0)}\in \mathbb{R}^n$.
\end{corollary}
When $B=I$ we can have the following Corollary \ref{cor:mgsor} directly from Corollary \ref{cor:exmgsor}.
\begin{corollary}\label{cor:mgsor}
Assume that $A,Q_{\rm MGSOR}\in \mathbb{R}^{n\times n}$ are given nonsingular matrices, $\theta=\beta_{\rm MGSOR}>0$ and $\alpha_{\rm MGSOR}>0$. If
\begin{subequations}\label{ie:mgsor2}
\begin{equation}
|1-\alpha_{\rm MGSOR}||1-\beta_{\rm MGSOR}|<1
\end{equation}
\text{and}
\begin{equation}
\alpha_{\rm MGSOR}\beta_{\rm MGSOR}\Vert A^{-1}Q_{\rm MGSOR}\Vert\Vert Q_{\rm MGSOR}^{-1}\Vert<(1-|1-\alpha_{\rm MGSOR}|)(1-|1-\beta_{\rm MGSOR}|),
\end{equation}
\end{subequations}
then AVEs \eqref{eq:ave} has a unique solution $x^*$ for any $b\in \mathbb{R}^n$ and the sequence $\{(x^{(k)}, y^{(k)})\}_{k=0}^{\infty}$ generated by the MGSOR iteration \eqref{eq:mgsor} converges to $(x^*,y^*=Q_{\rm MGSOR}^{-1}|x^*|)$ for any given initial vectors $x^{(0)},y^{(0)}\in \mathbb{R}^n$.
\end{corollary}

\begin{remark}{\rm
The convergence condition \eqref{ie:mgsor2} is new, which is not appeared in \cite{zhzl2023}. In addition, the existence and uniqueness of the solution to AVEs~\eqref{eq:ave} is not considered in \cite{zhzl2023}.}
\end{remark}

When $\alpha_{\rm MGSOR}=1$, $\beta_{\rm MGSOR} = \tau_{\rm MFPI}$ and $Q_{\rm MGSOR}=Q_{\rm MFPI}$, the following Corollary~\ref{cor:exmfpi} follows from Corollary~\ref{cor:exmgsor}.
\begin{corollary}\label{cor:exmfpi}
Assume that $A, Q_{\rm MFPI}\in \mathbb{R}^{n\times n}$ are given nonsingular matrices and $B\in\mathbb{R}^{n\times n}$ is given, $\theta = \tau_{\rm MFPI}>0$. If
\begin{equation*}
0< \Vert A^{-1}BQ_{\rm MFPI}\Vert\Vert Q^{-1}_{\rm MFPI}\Vert <\frac{1-|1-\tau_{\rm MFPI}|}{\tau_{\rm MFPI}},
\end{equation*}
then GAVEs \eqref{eq:gave} has a unique solution $x^*$ for any $b\in \mathbb{R}^n$ and the sequence $\{(x^{(k)}, y^{(k)})\}_{k=0}^{\infty}$ generated by the MFPI iteration \eqref{eq:mfpi4gave} converges to $(x^*,y^*=Q_{\rm MFPI}^{-1}|x^*|)$ for any given initial vectors $x^{(0)},y^{(0)}\in \mathbb{R}^n$.
\end{corollary}
When $\alpha_{\rm MGSOR}=1$, $\beta_{\rm MGSOR} = \tau_{\rm MFPI}$, $Q_{\rm MGSOR}=Q_{\rm MFPI}$ and $B=I$ we can establish the following Corollary \ref{cor:mfpi} from Corollary \ref{cor:mgsor}.
\begin{corollary}\label{cor:mfpi}
Suppose $A, Q_{\rm MFPI}\in \mathbb{R}^{n\times n}$ are given nonsingular matrices and $\theta = \tau_{\rm MFPI}>0$. If
\begin{equation}\label{ie:mfpi}
0< \Vert A^{-1}Q_{\rm MFPI}\Vert\Vert Q^{-1}_{\rm MFPI}\Vert <1\quad \text{and}\quad 0<\tau_{\rm MFPI}<\frac{2}{1 + \Vert A^{-1}Q_{\rm MFPI}\Vert\Vert Q^{-1}_{\rm MFPI}\Vert},
\end{equation}
 then AVEs \eqref{eq:ave} has a unique solution $x^*$ for any $b\in \mathbb{R}^n$ and the sequence $\{(x^{(k)}, y^{(k)})\}_{k=0}^{\infty}$ generated by the MFPI iteration \eqref{eq:mfpi} converges to $(x^*,y^*=Q_{\rm MFPI}^{-1}|x^*|)$ for any given initial vectors $x^{(0)},y^{(0)}\in \mathbb{R}^n$.
\end{corollary}

\begin{remark}{\rm
In \cite{yuch2022}, the authors show that the sequence $\{x^{(k)}\}_{k=0}^{\infty}$ generated by the MFPI iteration \eqref{eq:mfpi} converges to the unique solution $x^*$ of AVEs \eqref{eq:ave} if
\begin{equation}\label{ie:omfpi}
0<\nu <\dfrac{1}{\ell\sqrt{1+\varrho^2}}\quad \text{and}\quad
  \dfrac{1-\sqrt{1-\ell^2\nu^2}}{1-\ell\varrho\nu}<\tau_{\rm MFPI}<\dfrac{1+\sqrt{1-\ell^2\nu^2}}{1+\ell\varrho\nu},
\end{equation}
where $\nu = \Vert A^{-1}\Vert, \varrho=\Vert Q^{-1}_{\rm MFPI}\Vert, \ell=\Vert Q_{\rm MFPI}\Vert$.

It follows from $0<\nu <\dfrac{1}{\ell\sqrt{1+\varrho^2}}$ that
$$0<\Vert A^{-1}Q_{\rm MFPI}\Vert\Vert Q^{-1}_{\rm MFPI}\Vert\le \Vert A^{-1}\Vert \Vert Q_{\rm MFPI}\Vert\Vert Q^{-1}_{\rm MFPI}\Vert <1.$$
In addition, $0<\nu <\dfrac{1}{\ell\sqrt{1+\varrho^2}}$ implies  $\dfrac{1-\sqrt{1-\ell^2\nu^2}}{1-\ell\varrho\nu}>0$ and $\frac{2}{1 + \Vert A^{-1}Q_{\rm MFPI}\Vert\Vert Q^{-1}_{\rm MFPI}\Vert} \ge \frac{2}{1 + \Vert A^{-1}\Vert \Vert Q_{\rm MFPI}\Vert\Vert Q^{-1}_{\rm MFPI}\Vert} > \dfrac{1+\sqrt{1-\ell^2\nu^2}}{1+\ell\varrho\nu}$. Hence, \eqref{ie:mfpi} is weaker than \eqref{ie:omfpi}.}
\end{remark}

If $Q_{\rm MFPI}= I$ and $\tau_{\rm MFPI} = \tau_{\rm FPI}$, the following corollary follows from Corollary \ref{cor:exmfpi}.
\begin{corollary}\label{cor:exfpi1}
Let $A,B\in\mathbb{R}^{n\times n}$ be given and assume that $A$ is nonsingular, $\theta=\tau_{\rm FPI}>0$. If
\begin{equation*}
\|A^{-1}B\|<\frac{|1-\tau_{\rm FPI}|}{\tau_{\rm FPI}}
\end{equation*}
then GAVEs \eqref{eq:gave} has a unique solution $x^*$ for any $b\in \mathbb{R}^n$ and the sequence $\{(x^{(k)}, y^{(k)})\}_{k=0}^{\infty}$ generated by the FPI iteration \eqref{eq:fpi4gave} converges to $(x^*,y^*=|x^*|)$ for any given initial vectors $x^{(0)},y^{(0)}\in \mathbb{R}^n$.
\end{corollary}
When $Q_{\rm MFPI}= I$, $\tau_{\rm MFPI} = \tau_{\rm FPI}$ and $B=I$ we can have the following result follows from Corollary~\ref{cor:mfpi}.
\begin{corollary}\label{cor:fpi1}
Let $A\in \mathbb{R}^{n\times n}$ be a given nonsingular matrix and $\nu=\|A^{-1}\|$. If
\begin{equation}\label{eq:fpiconv1}
0<\nu<1 \quad \text{and} \quad 0<\tau_{\rm FPI}<\frac{2}{1+\nu},
\end{equation}
then AVEs \eqref{eq:ave} has a unique solution $x^*$ for any $b\in \mathbb{R}^n$ and the sequence $\{(x^{(k)}, y^{(k)})\}_{k=0}^{\infty}$ generated by the FPI iteration \eqref{eq:fpi} converges to $(x^*,y^*=|x^*|)$ for any given initial vectors $x^{(0)},y^{(0)}\in \mathbb{R}^n$.
\end{corollary}

\begin{remark}{\rm
In \cite{ke2020},  Ke show that the sequence  $\{x^{(k)}\}_{k=0}^{\infty}$ generated by the FPI iteration~\eqref{eq:fpi} converges to the unique solution  of AVEs~\eqref{eq:ave} if
\begin{equation}\label{eq:fpi1}
0 < \nu <\frac{\sqrt{2}}{2} \quad \text{and} \quad \frac{1-\sqrt{1-\nu^2}}{1-\nu}<\tau_{\rm FPI}<\frac{1+\sqrt{1-\nu^2}}{1+\nu}.
\end{equation}
It is obvious that $0< \dfrac{1-\sqrt{1-\nu^2}}{1-\nu}$ and $\dfrac{1+\sqrt{1-\nu^2}}{1+\nu}<\dfrac{2}{1+\nu}$ when $0<\nu <1$. Hence, $\{\nu: 0<\nu < \frac{\sqrt{2}}{2}\} \subseteq \{\nu: 0<\nu < 1\}$ and $\{\tau_{\rm FPI}: \frac{1-\sqrt{1-\nu^2}}{1-\nu}<\tau_{\rm FPI}<\frac{1+\sqrt{1-\nu^2}}{1+\nu}\} \subseteq \{\tau_{\rm FPI}: 0<\tau_{\rm FPI} < \frac{2}{1+ \nu}\}$. In other words,  \eqref{eq:fpiconv1} is weaker than \eqref{eq:fpi1}. Moreover, we get the theoretical improvement here without using the modification strategy proposed in \cite{yuch2022}.}
\end{remark}

When $\alpha_{\rm MGSOR}=\beta_{\rm MGSOR} = \omega_{\rm SOR}$ and $Q_{\rm MGSOR}=I$, the following corollary follows from Corollary \ref{cor:exmgsor}.
\begin{corollary}\label{cor:exsor}
Let $A,B\in \mathbb{R}^{n\times n}$ be given and assume that $A$ is invertible, $\theta =\omega_{\rm SOR}>0$. If
\begin{equation*}
0<\omega_{\rm SOR} <2 \quad \text{and} \quad 0< \|A^{-1}B\|<\Big(\dfrac{1-|1-\omega_{\rm SOR}|}{\omega_{\rm SOR}}\Big)^2,
\end{equation*}
then GAVEs \eqref{eq:gave} has a unique solution $x^*$ for any $b\in \mathbb{R}^n$ and the sequence $\{(x^{(k)}, y^{(k)})\}_{k=0}^{\infty}$ generated by the SOR-like iteration \eqref{it:sor4gave} converges to $(x^*,y^*=|x^*|)$ for any given initial vectors $x^{(0)},y^{(0)}\in \mathbb{R}^n$.
\end{corollary}
When  $\alpha_{\rm MGSOR}=\beta_{\rm MGSOR} = \omega_{\rm SOR}$, $Q_{\rm MGSOR}=I$ and $B=I$ we can have the following result from Corollary \ref{cor:mgsor}.
\begin{corollary}\label{cor:sor}
Let $A\in \mathbb{R}^{n\times n}$ be a given nonsingular matrix, $\nu= \|A^{-1}\|$ and $\theta = \omega_{\rm SOR}>0$. If
\begin{equation}\label{ie:sor}
0< \nu <1 \quad \text{and} \quad 0< \omega_{\rm SOR}<\frac{2-2\sqrt{\nu}}{1-\nu},
\end{equation}
then AVEs \eqref{eq:ave} has a unique solution $x^*$ for any $b\in \mathbb{R}^n$ and the sequence $\{(x^{(k)}, y^{(k)})\}_{k=0}^{\infty}$ generated by the SOR-like iteration \eqref{it:sor} converges to $(x^*,y^*=|x^*|)$ for any given initial vectors $x^{(0)},y^{(0)}\in \mathbb{R}^n$.
\end{corollary}

\begin{remark}\label{rem:SOR}{\rm
In \cite[Corollary 3.1]{KM2017}, under a kernel norm,  Ke and Ma show that the SOR-like iteration converges to the unique solution of GAVEs~\eqref{eq:gave} if
\begin{equation}\label{ie:osor}
0<\nu <1 \quad \text{and} \quad 1- \frac{2}{3+ \sqrt{5}}< \omega_{\rm SOR}< \min\left\{ 1 +\frac{2}{3+ \sqrt{5}}, \sqrt{\frac{2}{(3+\sqrt{5})\nu}}\right\}.
\end{equation}
If $0< \nu < \left(\frac{1+\sqrt{5}}{5+  \sqrt{5}}\right)^2$, we have $\frac{2-2\sqrt{\nu}}{1-\nu} > 1 +\frac{2}{3+ \sqrt{5}} \ge \min\left\{ 1 +\frac{2}{3+ \sqrt{5}}, \sqrt{\frac{2}{(3+\sqrt{5})\nu}}\right\}$, which implies that \eqref{ie:sor} is weaker than \eqref{ie:osor}.}
\end{remark}

For the MSOR iteration \cite{huli2022}, we have the following argument.

\begin{corollary}\label{cor:exmsor}
Let $A,B\in\mathbb{R}^{n\times n}$ be two given nonsingular matrices, $0<\theta=\omega_{\rm MSOR}\neq 2$ and $\alpha_{\rm MSOR}>0$. If
\begin{align*}
\bigg| |1-\alpha_{\rm MSOR}\omega_{\rm MSOR}||1-\omega_{\rm MSOR}|&+\omega_{\rm MSOR}|1-\alpha_{\rm MSOR}\omega_{\rm MSOR}|\left\|I-\frac{2}{2-\omega_{\rm MSOR}}B\right\|\\
&-\frac{2\omega_{\rm MSOR}|1-\alpha_{\rm MSOR}\omega_{\rm MSOR}|}{|2-\omega_{\rm MSOR}|}\|A^{-1}B\|\|A\|\bigg|<1
\end{align*}
and
\begin{align*}
&\frac{2\omega_{\rm MSOR}}{|2-\omega_{\rm MSOR}|}\|A^{-1}B\|(\omega_{\rm MSOR}\alpha_{\rm MSOR}\|B\|+2|1-\alpha_{\rm MSOR}\omega_{\rm MSOR}|\|A\|)\\
 &\qquad <(1-|1-\alpha_{\rm MSOR}\omega_{\rm MSOR}|)
 \left(1-|1-\omega_{\rm MSOR}|-\omega_{\rm MSOR}\left\|I-\frac{2}{2-\omega_{\rm MSOR}}B\right\|\right),
\end{align*}
then GAVE \eqref{eq:gave} has a unique solution $x^*$ and the sequence $\{x^k\}_{k=0}^{\infty}$ generated by the MSOR iteration \eqref{eq:msor4gave} converges to $x^*$.
\end{corollary}

We can have the following result from Corollary \ref{cor:exmsor} with $B=I$.
\begin{corollary}
Let $A\in\mathbb{R}^{n\times n}$ be a given nonsingular matrix and $0<\theta=\omega_{\rm MSOR}\neq 2,~\alpha_{\rm MSOR}>0$. If
\begin{align}\nonumber
\bigg| |1-\alpha_{\rm MSOR}\omega_{\rm MSOR}||1-\omega_{\rm MSOR}|&+\omega_{\rm MSOR}|1-\alpha_{\rm MSOR}\omega_{\rm MSOR}|\Big|1-\frac{2}{2-\omega_{\rm MSOR}}\Big|\\\label{ie:msor1}
&-\frac{2\omega_{\rm MSOR}|1-\alpha_{\rm MSOR}\omega_{\rm MSOR}|}{|2-\omega_{\rm MSOR}|}\|A^{-1}\|\|A\|\bigg|<1
\end{align}
and
\begin{align}\nonumber
&\frac{2\omega_{\rm MSOR}}{|2-\omega_{\rm MSOR}|}\|A^{-1}\|\Big(\omega_{\rm MSOR}\alpha_{\rm MSOR}+2|1-\alpha_{\rm MSOR}\omega_{\rm MSOR}|\|A\|\Big)\\\label{ie:msor2}
&\qquad <\Big(1-|1-\alpha_{\rm MSOR}\omega_{\rm MSOR}|\Big)
\Big(1-|1-\omega_{\rm MSOR}|-\omega_{\rm MSOR}|1-\frac{2}{2-\omega_{\rm MSOR}}|\Big),
\end{align}
then AVE \eqref{eq:ave} has a unique solution $x^*$ for any $b\in \mathbb{R}^n$ and the sequence $\{x^k\}_{k=0}^{\infty}$ generated by the MSOR iteration \eqref{eq:socmsor} converges to $x^*$.
\end{corollary}
\begin{remark}{\rm
 The conditions \eqref{ie:msor1} and \eqref{ie:msor2} are different from the conditions of \cite[Theorem 3.2]{huli2022} and the unique existence of the solution of
 AVEs~\eqref{eq:ave} is not explored there.}
\end{remark}

For the NSOR iteration, we can have the following result.

\begin{corollary}\label{cor:nsor4gave}
Let $A,B\in\mathbb{R}^{n\times n}$ be two given nonsingular matrices, $\omega_{\rm NSOR}>0,~\sigma_{\rm NSOR}>0$. If
\begin{subequations}\label{eq:nsor1}
\begin{equation}
\Big||1-\omega_{\rm NSOR}|\| I - \frac{\omega_{\rm NSOR}^2}{\sigma_{\rm NSOR}}B\| - \frac{\omega_{\rm NSOR}^2|1-\omega_{\rm NSOR}|}{\sigma_{\rm NSOR}}\|A^{-1}B\|\|A\|\Big|<1~
\end{equation}
and
\begin{equation}
\frac{\omega_{\rm NSOR}^2}{\sigma_{\rm NSOR}}\|A^{-1}B\|\Big(\|B\| + |1-\omega_{\rm NSOR}|\| A\|\Big) <\Big(1-|1-\omega_{\rm NSOR}|\Big)\Big(1-\| I - \frac{\omega_{\rm NSOR}^2}{\sigma_{\rm NSOR}}B\|\Big),
\end{equation}
\end{subequations}
then GAVE \eqref{eq:gave} has a unique solution $x^*$ for any $b\in \mathbb{R}^n$ and the sequence $\{x^k\}_{k=0}^{\infty}$ generated by the NSOR iteration \eqref{eq:nsor4gave} converges to $x^*$.
\end{corollary}
When $B=I$, the Corollary \ref{cor:nsor4gave} turns into Corollary \ref{cor:nsor}.
\begin{corollary}\label{cor:nsor}
Let $A\in\mathbb{R}^{n\times n}$ be a given nonsingular matrix and $\theta=\omega_{\rm NSOR}>0,~\sigma_{\rm NSOR}>0$. If
\begin{subequations}
\begin{equation}\label{ie:nsor1}
\Big||1-\omega_{\rm NSOR}|| 1 - \frac{\omega_{\rm NSOR}^2}{\sigma_{\rm NSOR}}| - \frac{\omega_{\rm NSOR}^2|1-\omega_{\rm NSOR}|}{\sigma_{\rm NSOR}}\|A^{-1}\|\|A\|\Big|<1~
\end{equation}
\text{and}
\begin{equation}\label{ie:nsor2}
\frac{\omega_{\rm NSOR}^2}{\sigma_{\rm NSOR}}\|A^{-1}\|\Big(1 + |1-\omega_{\rm NSOR}|\| A\|\Big) <\Big(1-|1-\omega_{\rm NSOR}|\Big)\Big(1-| 1 - \frac{\omega_{\rm NSOR}^2}{\sigma_{\rm NSOR}}|\Big),
\end{equation}
\end{subequations}
 then AVEs \eqref{eq:ave} has a unique solution $x^*$ for any $b\in \mathbb{R}^n$ and the sequence $\{x^k\}_{k=0}^{\infty}$ generated by the NSOR iteration \eqref{eq:nsor} converges to $x^*$.
\end{corollary}
\begin{remark}{\rm
The convergence conditions \eqref{ie:nsor1} and \eqref{ie:nsor2} are new, which  are not appeared in \cite{doss2020}. In addition, the existence and uniqueness of the solution to AVEs~\eqref{eq:ave} is not explored in \cite{doss2020}.}
\end{remark}

For FPI-SS iteration \cite{lild2022}, the following corollary is obtained from Corollary \ref{cor:con1}.

\begin{corollary}\label{cor:fpiss}
  Assume that $A,B\in \mathbb{R}^{n\times n}$ are given and $\alpha_{\rm FPI\text{-}SS} I+A$ is nonsingular with $\alpha_{\rm FPI\text{-}SS}>0$ and $\theta = \omega_{\rm FPI\text{-}SS}$. Denote $\hat{a}=\Vert (\alpha_{\rm FPI\text{-}SS} I+A)^{-1}(\alpha_{\rm FPI\text{-}SS} I - A)\Vert,~\hat{b}=2\Vert (\alpha_{\rm FPI\text{-}SS} I+A)^{-1}B\Vert.$
If
\begin{equation}\label{eq:fpiss1}
\hat{a}+ \hat{b}<1 \quad \text{and} \quad 0<\omega_{\rm FPI\text{-}SS}<\frac{2(1-\hat{a})}{1-\hat{a}+\hat{b}},
\end{equation}
then GAVEs \eqref{eq:gave} has a unique solution $x^*$ for any $b\in \mathbb{R}^n$ and the sequence $\{(x^{(k)}, y^{(k)})\}_{k=0}^{\infty}$ generated by the FPI-SS iteration \eqref{eq:fpiss} converges to $(x^*,y^*=|x^*|)$ for any given initial vectors $x^{(0)},y^{(0)}\in \mathbb{R}^n$.
\end{corollary}

\begin{remark}\label{rem:FPISS}{\rm
In \cite[Theorem 1]{lild2022}, under the assumption that GAVEs~\eqref{eq:gave} has a unique solution, the sequence $\{x^{(k)}\}_{k=0}^{\infty}$ generated by the FPI-SS iteration converges to the unique solution $x^*$ of GAVEs~\eqref{eq:gave} if
\begin{equation}\label{eq:fpiss4}
\hat{a}+\hat{b}<1\quad \text{and} \quad 0<\omega_{\rm FPI\text{-}SS}<\dfrac{2}{1+\hat{a}+\hat{b}}.
\end{equation}
It is easy to see that $\dfrac{2}{1+\hat{a}+\hat{b}}<\dfrac{2(1-\hat{a})}{1-\hat{a}+\hat{b}}$,  which implies that \eqref{eq:fpiss1} is weaker
than~\eqref{eq:fpiss4}. Furthermore we can conclude that GAVEs \eqref{eq:gave} has a unique solution under condition~\eqref{eq:fpiss4}.}
\end{remark}

\section{Comparison theorems}\label{sec:compar}
It follows from \eqref{eq:cr} that the smaller value of $\rho(T)$ is (where $T$ is defined as in \eqref{eq:t}), the faster the GRMS
iteration~\eqref{eq:grms} will converge later on. In this section, we will theoretically compare the convergence rate of the GRMS iteration \eqref{eq:grms} with some existing iterations, i.e., the RMS iteration \eqref{eq:rms}, the MGSOR iteration~\eqref{eq:mgsor4gave} and the NSOR iteration \eqref{eq:nsor4gave}. To this end, we introduce the following two lemmas.

\begin{lem}[{\cite[p. 27]{bepl1994}}]\label{lem:irr}
 For $n\geq2$, a nonnegative matrix $U\in\mathbb{R}^{n\times n}$ is irreducible if $(I+U)^{n-1}$ is a positive matrix, i.e. $(I+U)^{n-1}>0$.
\end{lem}
\begin{lem}[{\cite[p. 27]{bepl1994}}]\label{lem:rho}
 Let $U,R\in\mathbb{R}^{n\times n}$. If $R+U$ is an irreducible matrix, $U\neq R$ and $U\geq R\geq0$, $\rho(U)>\rho(R)$.
\end{lem}

\subsection{GRMS vs. RMS}
Let
$$
T_{\rm RMS}= \begin{bmatrix}
        \|M^{-1}_{\rm RMS}N_{\rm RMS}\| & \|M^{-1}_{\rm RMS}B\| \\
        \tau_{\rm RMS}\|M^{-1}_{\rm RMS}N_{\rm RMS}\| & |1-\tau_{\rm RMS}| + \tau_{\rm RMS} \|M^{-1}_{\rm RMS}B\| \\
      \end{bmatrix},
$$
which is derived from \eqref{eq:t} with $M=M_{\rm RMS}$, $Q=I$, $Q_1=\frac{\theta}{\tau_{\rm RMS}}I$, $H=0$ and $\theta=\tau_{\rm RMS}$.
Similar to the proof of Theorem \ref{thm:cov}, the smaller value of $\rho(T_{\rm RMS})$ is, the faster the RMS iteration~\eqref{eq:rms}
will converge later on. For GRMS and RMS, we have the following comparison theorem.

\begin{theorem}\label{thm:com}
 If $A,B\in\mathbb{R}^{n\times n}$ with $A=M-N$ such that $M$ is nonsingular and $M_{\rm RMS}=M$. $Q=Q_1-Q_2\in\mathbb{R}^{n\times n}$ and $Q_1$ is nonsingular. $\tau_{\rm RMS}>0$ and $\theta>0$. Assume that notation \eqref{eq:nota} and \cite[Lemma~3]{soso2023} hold, and $T$ is an irreducible matrix. If
\begin{subequations}\label{eqeq:conv}
\begin{equation}\label{eq:conv1}
    \|Q\|\leq1,~\beta c<|1 - \tau_{\rm RMS}| + \tau_{\rm RMS} \|M^{-1}B\| + 1,~\beta(a+1)<\tau_{\rm RMS} a, ~\alpha c + d <1,
\end{equation}
\begin{equation}
(1 - d)[\tau_{\rm RMS} a - \beta(a+1)]>\alpha c(\tau_{\rm RMS} a - \beta) + a\alpha[1-|1-\tau_{\rm RMS}| - \tau_{\rm RMS} \|M^{-1}B\|]
\end{equation}
\text{and}
\begin{equation}\label{eq:conv2}\footnotesize
\dfrac{1 + \beta c - (|1 - \tau_{\rm RMS}| + \tau_{\rm RMS} \|M^{-1}B\|)}{1 - (\alpha c + d)}<\theta<\min\left\{\dfrac{|1 - \tau_{\rm RMS}| + \tau_{\rm RMS}\|M^{-1}B\| + 1 - \beta c}{1 + \alpha c + d},~\frac{\tau_{\rm RMS} a - (a+1)\beta}{a\alpha}\right\},
\end{equation}
\end{subequations}
then $\rho(T)<\rho(T_{\rm RMS})<1$.
\end{theorem}
\begin{proof}
According to Lemma \ref{lem:rho}, condition \eqref{eqeq:conv} and the assumption that \cite[Lemma~3]{soso2023} hold, we know that $0\leq T\leq T_{\rm RMS}$ and $T\neq T_{\rm RMS}$, which implied that $\rho(T)<\rho(T_{\rm RMS})<1$ since $T$ is an irreducible matrix.
\end{proof}
The following example is used to verify Theorem \ref{thm:com}.

\begin{example}\label{ex:thm}
Consider GAVEs~\eqref{eq:gave} with $A=M-N$, $B = \begin{bmatrix}2 & -1.5 & 1.5 & 0 \\-1.5 & 2 & 0 & 1.5 \\1.5 & 0 & 2 & -1.5 \\0 & 1.5 & -1.5 & 2 \\\end{bmatrix} $ and $b = Ax^*-B|x^*|$ where $x^* = [-\frac{1}{2},0,-\frac{1}{2},0]^\top$.
\begin{enumerate}[(i)]
  \item For GRMS iteration: $M = \begin{bmatrix}8 & -1.5 & 1.5 & 0 \\-1.5 & 8 & 0 & 1.5 \\1.5 & 0 & 8 & -1.5 \\0 & 1.5 & -1.5 & 8 \\ \end{bmatrix}$, $N =-\dfrac{1}{8}I$, $Q = Q_1 - Q_2=0.98I$ with $Q_1=0.97I$, $H = 0.01I$ and $\theta = 0.66$. Then $T \approx \begin{bmatrix}0.0250 & 0.4455 \\0.0276 & 0.6545 \end{bmatrix}$ is an irreducible matrix and $\rho(T)\approx 0.6734$;
  \item For RMS iteration: $M_{\rm RMS}=M$, $N_{\rm RMS}=N$ and $\tau_{\rm RMS}=1.21$. Then $T_{\rm RMS}\approx\begin{bmatrix}0.0250 & 0.4545\\0.0302 & 0.7600 \end{bmatrix}$ and $\rho(T_{\rm RMS})\approx0.7783$.
\end{enumerate}

 It's easy to verify that the conditions of Theorem \ref{thm:com} hold and $\rho(T)<\rho(T_{\rm RMS})<1$. By setting initial vector $x^{(0)} = y^{(0)} = [0,0,0,0]^\top$, we can see that the IT {\rm (}the number of iteration{\rm)} of GRMS is less than that of RMS from Figure \ref{fig:rms}.
\end{example}

\begin{figure}[htp]
  \centering
  \includegraphics[width=0.9\textwidth]{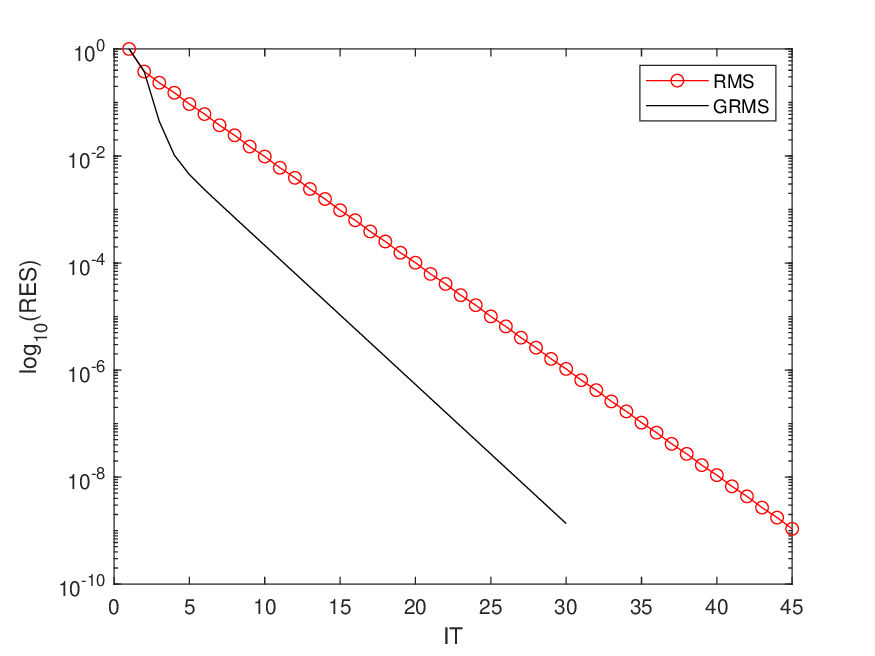}\\
  \caption{The curve of RES:=$\frac{\|Ax^{(k)} - B|x^{(k)}| -b\|_2}{\|b\|_2}$ and IT (the number of iteration) between GRMS and RMS iteration. The methods are terminated when ${\rm RES} \leq 10^{-9}$.}\label{fig:rms}
\end{figure}

\subsection{ GRMS vs. MGSOR}
Let
\begin{equation*}\footnotesize
T_{\rm MGSOR} =
  \begin{bmatrix}
    |1-\alpha_{\rm MGSOR}| & \alpha_{\rm MGSOR}\|A^{-1}BQ_{\rm MGSOR}\| \\
    \beta_{\rm MGSOR}|1-\alpha_{\rm MGSOR}|\|Q^{-1}_{\rm MGSOR}\| & \alpha_{\rm MGSOR}\beta_{\rm MGSOR}\|A^{-1}BQ_{\rm MGSOR}\|\|Q_{\rm MGSOR}^{-1}\|+|1-\beta_{\rm MGSOR}| \\
  \end{bmatrix},
\end{equation*}
which is derived from \eqref{eq:t} with $M=\frac{1}{\alpha_{\rm MGSOR}}A$, $Q=Q_{\rm MGSOR}, Q_1=\frac{\theta}{\beta_{\rm MGSOR}}Q_{\rm MGSOR}$, $H=0$ and $\theta=\beta_{\rm MGSOR}$. Similar to the proof of Theorem \ref{thm:cov}, the smaller value of $\rho(T_{\rm MGSOR})$ is, the faster the MGSOR
iteration~\eqref{eq:mgsor} will converge later on. For GRMS and MGSOR, we have the following comparison theorem.

\begin{theorem}\label{thm:mgsor}
Let $A,B\in\mathbb{R}^{n\times n}$ with $A=M-N$ such that $M$ is nonsingular. $Q_{\rm MGSOR}\in\mathbb{R}^{n\times n}$, $Q=Q_1-Q_2\in\mathbb{R}^{n\times n}$ and $Q_1$ is nonsingular. $\alpha_{\rm MGSOR}, \beta_{\rm MGSOR}$ and $\theta>0$. Assume that notation \eqref{eq:nota} and condition \eqref{ie:mgsor1} hold, and $T$ is an irreducible matrix.

If
\begin{subequations}\label{eq:covmgsor}
\begin{equation}
    a<|1- \alpha_{\rm MGSOR}|,~c<\alpha_{\rm MGSOR}\|A^{-1}BQ_{\rm MGSOR}\|,~\alpha c + d <1
\end{equation}
\text{and}
\begin{align}
&\beta \big((a+1)(1-d) - \alpha c\big)-\beta_{\rm MGSOR}|1-\alpha_{\rm MGSOR}|\|Q^{-1}_{\rm MGSOR}\|(1-\alpha c-d)<  \\
&a\alpha\left[|1-\beta_{\rm MGSOR}|+\alpha_{\rm MGSOR}\beta_{\rm MGSOR}\|A^{-1}BQ_{\rm MGSOR}\|\|Q^{-1}_{\rm MGSOR}\|-1\right]
\end{align}
\text{such that}
\begin{equation}
\tilde{\mu}<\theta <\min\left\{\tilde{\nu},~\Delta\right\},
\end{equation}
\end{subequations}
where $$\tilde{\mu} = \frac{1 + \beta c - \left(|1 - \beta_{\rm MGSOR}| + \alpha_{\rm MGSOR}\beta_{\rm MGSOR}\|A^{-1}BQ_{\rm MGSOR}\|\|Q^{-1}_{\rm MGSOR}\|\right)}{ 1-(\alpha c + d)},$$~$$\tilde{\nu} = \frac{1 - \beta c + |1 - \beta_{\rm MGSOR}| + \alpha_{\rm MGSOR}\beta_{\rm MGSOR}\|A^{-1}BQ_{\rm MGSOR}\|\|Q^{-1}_{\rm MGSOR}\|}{1 + \alpha c + d}$$ and $$\Delta=\frac{\beta_{\rm MGSOR}|1-\alpha_{\rm MGSOR}|\|Q^{-1}_{\rm MGSOR}\| - (a+1)\beta}{a\alpha},$$
then $\rho(T)<\rho(T_{\rm MGSOR})<1$.
\end{theorem}
\begin{proof}
According to Lemma \ref{lem:rho}, condition \eqref{eq:covmgsor} and the assumption that  \eqref{ie:mgsor1} hold, we know that $0\leq T< T_{\rm MGSOR}$, which implied that $\rho(T)<\rho(T_{\rm MGSOR})<1$ since $T$ is an irreducible matrix.
\end{proof}

The following Example \ref{ex:mgsor} is used to verify Theorem \ref{thm:mgsor}.
\begin{example}\label{ex:mgsor}
Consider AVEs~\eqref{eq:ave} with $A=M-N$ and $b = Ax^*-|x^*|$ with $x^* = [\frac{1}{2},0,\frac{1}{2},0]^\top$.

\begin{enumerate}[(i)]
  \item For GRMS iteration: $M = \begin{bmatrix}24 & -1 & -1 & 0 \\-1 & 24 & 0 & -1 \\-1 & 0 & 24 & -1 \\0 & -1 & -1 & 24 \\\end{bmatrix}$, $N = \dfrac{1}{4}I$, $Q =0.982I$, $Q_1=0.981I$, $H = 0.01I$ and $\theta = 1$. Then $T \approx \begin{bmatrix}0.0114 & 0.0446 \\0.0219 & 0.0470 \end{bmatrix}$ is an irreducible matrix and $\rho(T)\approx 0.0651$;
  \item For MGSOR iteration: $Q_{\rm MGSOR}=Q$ and $\alpha_{\rm MGSOR} = 1.05,~\beta_{\rm MGSOR} = 1.07$. Then $T_{\rm MGSOR} \approx \begin{bmatrix}0.0500 & 0.0474 \\0.0545 & 0.1217 \end{bmatrix}$ and $\rho(T_{\rm MGSOR})\approx0.1480$.
\end{enumerate}

 We can verify that the conditions of Theorem~\ref{thm:mgsor} are satisfied and we have $\rho(T) <\rho(T_{\rm MGSOR})<1$. Set initial vectors $x^{(0)} = y^{(0)} = [0,0,0,0]^\top$. Numerical results for Example~\ref{ex:mgsor} are reported in Figure~\ref{fig:mgsor}, from which we can conclude that the number of iteration for the GRMS method is less than that of the MGSOR method.

\begin{figure}[htp]
  \centering
  \includegraphics[width=4.5in]{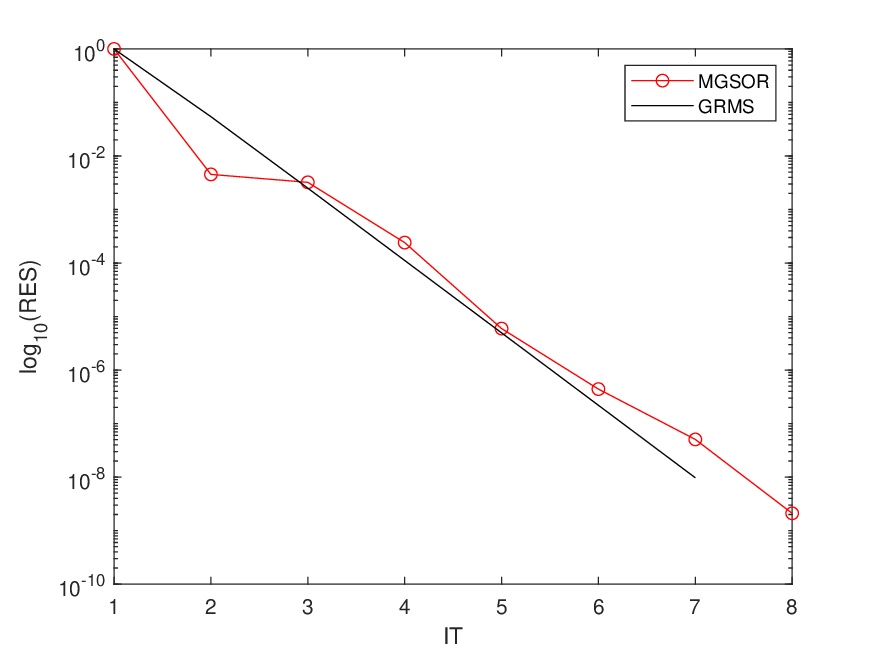}\\
  \caption{The curve of RES:=$\frac{\|Ax^{(k)} - |x^{(k)}| -b\|_2}{\|b\|_2}$ for Example~\ref{ex:mgsor}. The methods are terminated when ${\rm RES} \leq 10^{-9}$. }\label{fig:mgsor}
\end{figure}
\end{example}

\subsection{GRMS vs. NSOR}
For NSOR, similar to the proof of Theorem \ref{thm:cov}, the smaller value of $\rho(T_{\rm NSOR})$ is, the faster the NSOR iteration~\eqref{eq:nsor4gave} will converge later on, where
$$
T_{\rm NSOR}=
  \begin{bmatrix}
    |1-\omega_{\rm NSOR}| & \omega_{\rm NSOR}\|A^{-1}B\| \\
    \frac{\omega_{\rm NSOR}|1-\omega_{\rm NSOR}|}{\sigma_{\rm NSOR}}(\|A\|+\|B\|) & \frac{\omega_{\rm NSOR}^2}{\sigma_{\rm NSOR}}\|A^{-1}B\|\|B\|+\|I-\frac{\omega_{\rm NSOR}^2}{\sigma_{\rm NSOR}}B\| \\
  \end{bmatrix}.
$$
For GRMS and NSOR, we have the following comparison theorem.

\begin{theorem}\label{thm:nsor}
Let $A,B\in\mathbb{R}^{n\times n}$ with $A=M-N$ such that $M$ is nonsingular. $Q=Q_1-Q_2\in\mathbb{R}^{n\times n}$ and $Q_1$ is nonsingular. $\sigma_{\rm NSOR}$, $\omega_{\rm NSOR}>0$ and $\theta>0$. Let notation \eqref{eq:nota} and condition \eqref{eq:nsor1} hold, and $T$ is an irreducible matrix.

If
\begin{subequations}\label{conv:nsor12}
\begin{equation}
    a<|1- \omega_{\rm NSOR}|,~c<\omega_{\rm NSOR}\|A^{-1}B\|,~\alpha c + d <1
\end{equation}
\text{and}
\begin{align}\small
&\left[\frac{\omega_{\rm NSOR}|1-\omega_{\rm NSOR}|}{\sigma_{\rm NSOR}}\left(\|A\|+\|B\|\right)\right](\alpha c+d-1)-\beta\big[(\alpha+1)(d-1)+\alpha c\big]\\
&<a\alpha\left[\frac{\omega^2_{\rm NSOR}}{\sigma_{\rm NSOR}}\|A^{-1}B\|\|B\|+\|I-\frac{\omega^2_{\rm NSOR}}{\sigma_{\rm NSOR}}B\|-1\right]
\end{align}
\text{such that}
\begin{equation}
\hat{\mu}<\theta <\min\left\{\hat{\nu},~\Delta\right\},
\end{equation}
\end{subequations}
where $$\hat{\mu} = \frac{\frac{\omega^2_{\rm NSOR}}{\sigma_{\rm NSOR}}\|A^{-1}B\|\|B\|+\|I-\frac{\omega^2_{\rm NSOR}}{\sigma_{\rm NSOR}}B\|-\beta c-1}{\alpha c + d-1},$$~$$\hat{\nu} = \frac{\frac{\omega^2_{\rm NSOR}}{\sigma_{\rm NSOR}}\|A^{-1}B\|\|B\|+\|I-\frac{\omega^2_{\rm NSOR}}{\sigma_{\rm NSOR}}B\|-\beta c+1}{1 + \alpha c + d}$$ and $$\Delta=\frac{\frac{\omega_{\rm NSOR}|1-\omega_{\rm NSOR}|}{\sigma_{\rm NSOR}}\big(\|A\|+\|B\|\big) - \big(a+1\big)\beta}{a\alpha},$$
then $\rho(T)<\rho(T_{\rm NSOR})<1$.
\end{theorem}
\begin{proof}
According to Lemma \ref{lem:rho}, condition \eqref{conv:nsor12} and the assumption that \eqref{eq:nsor1} hold, we know that $0\leq T< T_{\rm NSOR}$, which implied that $\rho(T)<\rho(T_{\rm NSOR})<1$ since $T$ is an irreducible matrix.
\end{proof}
The Example \ref{ex:nsor} can be used to clarify Theorem \ref{thm:nsor}.
\begin{example}\label{ex:nsor}
Consider AVEs~\eqref{eq:ave} with $A=M-N$ and $b = Ax^*-|x^*|$ where $x^* = [\frac{1}{2},0,\frac{1}{2},0]^\top$.
\begin{enumerate}[(i)]
  \item For GRMS iteration: $M = \begin{bmatrix}4 & -0.5 & 0.5 & 0 \\-0.5 & 4 & 0 & 0.5 \\0.5 & 0 & 4 & -0.5 \\0 & 0.5 & -0.5 & 4 \\\end{bmatrix}$, $N = \dfrac{1}{10}I$, $Q =0.99I$, $Q_1=0.98I$, $H = 0.01I$ and $\theta = 1$. Then $T \approx \begin{bmatrix}0.0333 & 0.3300 \\0.0446 & 0.3503 \end{bmatrix}$ is an irreducible matrix and $\rho(T)\approx 0.3914$;
  \item For NSOR iteration: $\omega_{\rm NSOR} = 1.1,~\sigma_{\rm NSOR} = 0.98$. Then $T_{\rm NSOR}\approx \begin{bmatrix}0.1000 & 0.3793 \\0.6622 &0.6605 \end{bmatrix}$ and $\rho(T_{\rm NSOR})\approx 0.9544$.
\end{enumerate}

We can find that the conditions of Theorem~\ref{thm:nsor} hold, and $\rho(T)<\rho(T_{\rm NSOR})<1$. Set initial vector $x^{(0)} = y^{(0)} = [0,0,0,0]^\top$. From Figure \ref{fig:nsor}, we know that both the GRMS method and the NSOR method are convergent and the former is perform well than the latter in the cases of IT.

\begin{figure}[htp]
  \centering
  \includegraphics[width=0.9\textwidth]{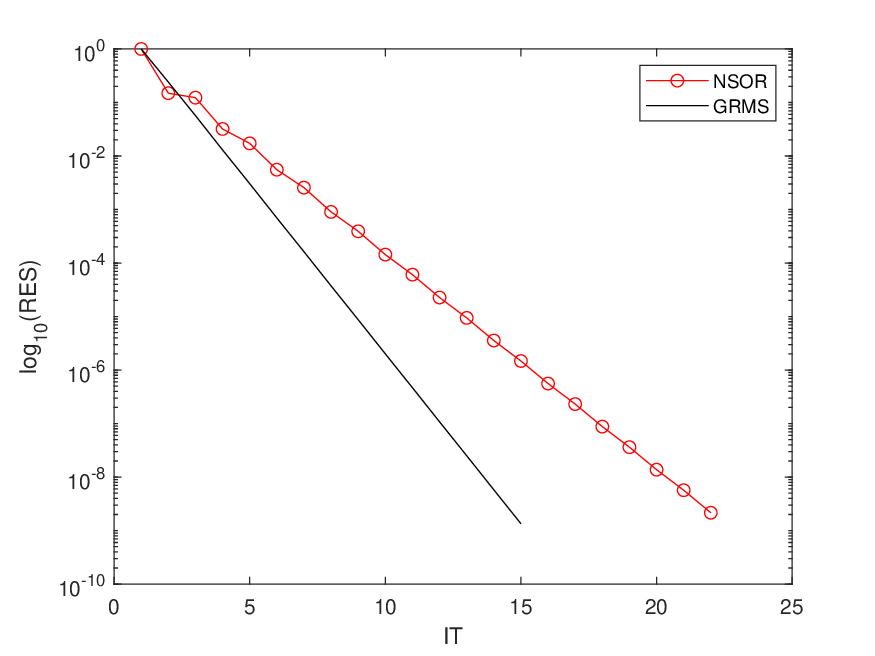}\\
  \caption{The curve of RES:=$\frac{\|Ax^{(k)} - |x^{(k)}| -b\|_2}{\|b\|_2}$ and IT (the number of iteration) between GRMS and NSOR iteration. The experiment is terminated if RES $\leq 10^{-9}$.}\label{fig:nsor}
\end{figure}

\end{example}

\section{ Numerical results}\label{sec:Numerical}
In this section, we will use two examples to numerically illustrate the performance of our method for solving GAVEs~\eqref{eq:gave}. All experiments are supported by a personal computer with $3.20$ GHZ central processing unit \big(Intel (R), Corel(TM), i$5$-$11320$H\big), $16$ GB memory and windows $11$ operating system, and implemented in MATLAB $9.10.0$ (R $2021$a). Nine iterative schemes will be tested.
\begin{enumerate}[]

  \item \textbf{GRMS}: The proposed iterative scheme \eqref{eq:grms} with $M= D-\dfrac{3}{4}L$, $Q = 10.5I$ and $Q_1 = 10 I$. In our experiments, we let $H = -\theta\zeta I$ with $\zeta \in \mathbb{R}$ being a constant.

  \item \textbf{Picard}: The Picard iterative scheme \eqref{eq:pi}.

  \item \textbf{SOR}: The SOR-like iterative scheme \eqref{it:sor4gave}.

  \item \textbf{MFPI}: The modified fixed point iterative scheme \eqref{eq:mfpi4gave} with $Q_{\rm MFPI}=Q$.

  \item \textbf{SSMN}: The shift-splitting Newton-type iterative scheme \eqref{eq:ssmm} with $\Omega_{\rm SSMN} = \text{diag}(A)$.

  \item \textbf{NSOR-like}: The Newton-based SOR iterative scheme \cite{zhwl2021}
  $$
  x^{k+1} = (D+\alpha\Omega_{\rm NSOR\text{-}like}-\alpha L)^{-1} \left[(\alpha\Omega_{\rm NSOR\text{-}like}+(1-\alpha)D+\alpha U) x^k + \alpha(B|x^k| + b)\right]
  $$
  with $\Omega_{\rm NSOR\text{-}like}=\text{diag}(A)$ and $\alpha >0$ being a constant.

  \item \textbf{MGSOR}: The modified generalized SOR-like iterative scheme \eqref{eq:mgsor4gave} with $Q_{\rm MGSOR}=Q$.

  \item \textbf{MAMS}: The momentum acceleration-based matrix splitting iterative scheme \eqref{eq:mams4gave} with $M_{\rm MAMS}=M$, $\Omega_{\rm MAMS}=\text{diag}(A)$ and the initial vectors $x^{(1)} =x^{(0)}$.

  \item \textbf{NSOR}: The new SOR-like iterative scheme \eqref{eq:nsor4gave}.
\end{enumerate}

All methods will be run ten times and the average IT (the number of iterations), the average CPU (the elapsed CPU time in seconds) and the average RES are reported, where
$$
{\rm RES} = \frac{\|Ax^{(k)} - B|x^{(k)}| -b\|}{\|b\|}.
$$
For Picard, SSMN, NSOR-like and MAMS, $x^{(0)}=(-\frac{1}{6},0,-\frac{1}{6},0,\dots,-\frac{1}{6},0)^\top$. For the other tested iterative schemes, $x^{(0)}=y^{(0)}=(-\frac{1}{6},0,-\frac{1}{6},0,\dots,-\frac{1}{6},0)^\top$. All methods are terminated when ${\rm RES} \le 10^{-8}$ or the maximal number of iterations $k_{\max} = 500$ is exceeded.

\begin{example}\label{exam1}{\rm
Consider GAVEs \eqref{eq:gave} with $A=\tilde{A}+\dfrac{1}{5}I$, where
\begin{equation*}
\scriptsize{
\setlength{\arraycolsep}{5pt}
\renewcommand{\arraystretch}{1}
\tilde{A}=
\left(
  \begin{array}{cccccccccccc}
    S_1 & -1.5I & -0.5I &-1.5I &-0.5I & 0 & 0 &0 & 0 & 0 & 0 & 0 \\
    -1.5I & S_1 & -1.5I &-0.5I &-1.5I & -0.5I & 0 &0 & 0 & 0 & 0 & 0\\
    -0.5I & -1.5I & S_1 &-1.5I &-0.5I & -1.5I & -0.5I &0 & 0 & 0 & 0 & 0\\
    -1.5I & -0.5I & -1.5I &S_1 &-1.5I & -0.5I & -1.5I &-0.5I & 0 & 0 & 0 & 0\\
    -0.5I & -1.5I & -0.5I &-1.5I &S_1 & -1.5I & -0.5I &-1.5I & -0.5I & 0 & 0 & 0\\
    0 & -0.5I & -1.5I &-0.5I &-1.5I & S_1 & -1.5I &-0.5I & -1.5I & -0.5I & 0 & 0\\
    \ddots&\ddots&\ddots&\ddots&\ddots&\ddots&\ddots
    &\ddots&\ddots&\ddots&\ddots&\ddots\\
    0 & 0&0&-0.5I & -1.5I &-0.5I &-1.5I & S_1 & -1.5I &-0.5I & -1.5I & -0.5I \\
    0 & 0&0&0&-0.5I & -1.5I &-0.5I &-1.5I & S_1 & -1.5I &-0.5I & -1.5I  \\
    0 & 0&0&0&0&-0.5I & -1.5I &-0.5I &-1.5I & S_1 & -1.5I &-0.5I   \\
    0 & 0&0&0&0&0&-0.5I & -1.5I &-0.5I &-1.5I & S_1 & -1.5I   \\
    0 & 0&0&0&0&0&0&-0.5I & -1.5I &-0.5I &-1.5I & S_1
  \end{array}
\right)
}
\in\mathbb{R}^{n\times n}
\end{equation*}
 and
$$B=
\left(
  \begin{array}{cccccccccccc}
    \tilde{S_2} & -I & -I &-I &-I & 0 & 0 &0 & 0 & 0 & 0 & 0 \\
    -I & \tilde{S_2} & -I &-I &-I & -I & 0 &0 & 0 & 0 & 0 & 0\\
    -I & -I & \tilde{S_2} &-I &-I & -I & -I &0 & 0 & 0 & 0 & 0\\
    -I & -I & -I &\tilde{S_2} &-I & -I & -I &-I & 0 & 0 & 0 & 0\\
    -I & -I & -I &-I &\tilde{S_2 }& -I & -I &-I & -I & 0 & 0 & 0\\
    0 & -I & -I &-I &-I & \tilde{S_2} & -I &-I & -I & -I & 0 & 0\\
    \ddots&\ddots&\ddots&\ddots&\ddots&\ddots&\ddots
    &\ddots&\ddots&\ddots&\ddots&\ddots\\
    0 & 0&0&-I & -I &-I &-I & \tilde{S_2} & -I &-I & -I & -I \\
    0 & 0&0&0&-I & -I &-I &-I & \tilde{S_2} & -I &-I & -I  \\
    0 & 0&0&0&0&-I & -I &-I &-I & \tilde{S_2} & -I &-I   \\
    0 & 0&0&0&0&0&-I & -I &-I &-I & \tilde{S_2} & -I   \\
    0 & 0&0&0&0&0&0&-I & -I &-I &-I & \tilde{S_2}
  \end{array}
\right)
\in\mathbb{R}^{n\times n},$$
 in which
$$
S_1=
\setlength{\arraycolsep}{0.9pt}
\left(
  \begin{array}{cccccccccc}
    36 & -1.5 & -0.5 &-1.5 & 0 &0 & 0 & 0 & 0 & 0 \\
    -1.5 & 36 & -1.5 &-0.5 &-1.5 &  0 &0 & 0 & 0 & 0 \\
    -0.5 & -1.5 & 36 &-1.5 &-0.5 & -1.5 & 0 & 0 & 0 & 0\\
    -1.5 & -0.5 & -1.5 &36 &-1.5 & -0.5 & -1.5 & 0 & 0 & 0\\
    0 & -1.5 & -0.5 &-1.5 &36 & -1.5 & -0.5 &-1.5 & 0 & 0\\
    0 & 0 & -1.5 &-0.5 &-1.5 & 36 & -1.5 &-0.5 & -1.5  & 0\\
    \ddots&\ddots&\ddots&\ddots&\ddots&\ddots&\ddots
    &\ddots&\ddots&\ddots\\
     0&0&0&0 & -1.5 &-0.5 &-1.5 & 36 & -1.5 &-0.5   \\
     0&0&0&0&0& -1.5 &-0.5 &-1.5 &36 & -1.5   \\
     0&0&0&0&0&0& -1.5 &-0.5 &-1.5 & 36
  \end{array}
\right)\in\mathbb{R}^{m\times m}
$$
and
$$\tilde{S_2}=
\setlength{\arraycolsep}{0.9pt}
\left(
  \begin{array}{cccccccccc}
    3 & -1 & -1 &-1 & 0 &0 & 0 & 0 & 0 & 0 \\
    -1 & 3 & -1 &-1 &-1 &  0 &0 & 0 & 0 & 0 \\
    -1 & -1 & 3 &-1 &-1 & -1 & 0 & 0 & 0 & 0\\
    -1 & -1 & -1 &3 &-1 & -1 & -1 & 0 & 0 & 0\\
    0 & -1 & -1 &-1 &3 & -1 & -1 &-1 & 0 & 0\\
    0 & 0 & -1 &-1 &-1 & 3 & -1 &-1 & -1  & 0\\
    \ddots&\ddots&\ddots&\ddots&\ddots&\ddots&\ddots
    &\ddots&\ddots&\ddots\\
     0&0&0&0 & -1 &-1 &-1 & 3 & -1 &-1   \\
     0&0&0&0&0& -1 &-1 &-1 &3 & -1   \\
     0&0&0&0&0&0& -1 &-1 &-1 & 3
  \end{array}
\right)
\in\mathbb{R}^{m\times m}.$$
In addition, $b=Ax^*-B|x^*|$ with $x^*=\left(\frac{1}{2},1,\frac{1}{2},1,\dots,\frac{1}{2},1\right)^\top$.

Numerical results for this example are reported in Table~\ref{table1}, in which $\theta_{opt}, \omega_{\rm opt}$, $\tau_{\rm opt}$, $\beta_{\rm opt}$, $\alpha_{\rm opt}$, $\sigma_{opt}$ and  $\zeta_{\rm opt}$ are numerically optimal iteration parameters of the tested methods, respectively\footnote{The same goes to Example~\ref{exam2}.}. For this example, we do not report the results of the NSOR iteration since it is not convergent whenever $m>2$. It follows from  Table~\ref{table1} that the GRMS iteration is superior to all of the other tested methods in terms of IT and CPU.}

\begin{table}[!h]\footnotesize \label{table1}
\centering
\caption{Numerical results for Example~\ref{exam1}.}\label{table1}
\begin{tabular}{ccccccc}\hline
			Method  & $m$ & $60$  & $80$ & $90$ & $100$ & $110$  \\ \hline
			\textbf{GRMS}   &  $(\theta_{\rm opt},\zeta_{\rm opt})$ & $(0.96,0.01)$ &  $(0.96,0.01)$  & $(0.96,0.01)$  & $(0.96,0.01)$ & $(0.96,0.01)$    \\
			&  IT  &  $\textbf{8}$ & $\textbf{8}$   & $\textbf{8}$    & $\textbf{8}$  &$\textbf{8}$   \\
			&   CPU & $\textbf{0.0028}$ &  $\textbf{0.0046}$ &  $\textbf{0.0067}$ & $\textbf{0.0123}$ &$\textbf{0.0157}$  \\
			&    RES &  2.7852e-09 & 2.1181e-09 & 1.8956e-09 & 1.7176e-09 & 1.5719e-09   \\
			Picard   &   &  &   &   &   &   \\
			&  IT  &  $27$ & $28$   & $28$    & $28$ & $28$   \\
			&   CPU & $0.2353$ &  $0.4274$ &  $0.6475$ & $0.9504$ &$1.2351$  \\
			&    RES &  9.4921e-09 & 5.8315e-09 & 6.1584e-09 & 6.4193e-09 & 6.6323e-09   \\
            SOR &   $\omega_{\rm opt}$ & $0.88$ & $0.88$  &  $0.88$ &  $0.90$ & $0.90$  \\
			&  IT  &  $12$ & $12$   & $12$    & $11$ & $11$   \\
			&   CPU & $0.1049$ &  $0.1935$ &  $0.2830$ & $0.3767$ &$0.4906$  \\
			&    RES &  8.4200e-09 & 7.2953e-09 & 6.8791e-09 & 9.6647e-09 & 9.3116e-09   \\
            MFPI   & $\tau_{\rm opt}$  & $0.77$ & $0.76$  &  $0.76$ &  $0.79$ & $0.79$  \\
			&  IT  &  $14$ & $14$   & $14$    & $13$ & $13$   \\
			&   CPU & $0.1256$ &  $0.2253$ &  $0.3401$ & $0.4489$ &$0.5989$  \\
			&    RES &  7.6818e-09 & 9.0837e-09 & 8.5772e-09 & 9.9758e-09 & 9.6026e-09   \\
            SSMN&  &  &  &  &  & \\
			&  IT  &  $9$ & $9$   & $9$  &  $9$ & $9$   \\
			&  CPU & $0.0814$ &  $0.1446$ &  $0.2180$ & $0.3124$ &$0.4075$   \\
			&  RES &  1.6836e-09  & 1.8415e-09  & 1.8937e-09  & 1.9353e-09  & 1.9693e-09  \\
            NSOR-like& $\alpha_{\rm opt}$ & $1.84$ & $1.99$  & $1.98$  & $1.97$  &$1.97$   \\
			&  IT  &  $18$ & $17$   & $17$    & $17$  &$17$   \\
			&  CPU & $0.0039$ &  $0.0059$ &  $0.0086$ & $0.0129$ &$0.0168$  \\
			&  RES &  9.8445e-09 & 9.9880e-09 & 9.9143e-09 & 9.9252e-09 & 9.4978e-09   \\
            MGSOR & $(\alpha_{\rm opt},\beta_{\rm opt})$ & $(0.93,0.87)$ & $(0.94,0.86)$ & $(0.94,0.86)$ & $(0.93,0.86)$ & $(0.95,0.85)$  \\
			&  IT  &  $12$ & $12$   & $12$    & $12$  &$12$   \\
			&  CPU & $0.1095$ &  $0.1916$ &  $0.2947$ & $0.4257$ &$0.5514$  \\
			&  RES &  9.5864e-09 & 9.5125e-09 & 9.0772e-09 & 9.5574e-09 & 9.9321e-09   \\
            MAMS & $\beta_{\rm opt}$ & $1.26$ & $1.20$  & $1.18$  & $2.00$  &$1.99$   \\
			&  IT  &  $27$ & $27$   & $27$    & $26$  &$26$   \\
			&  CPU & $0.0058$ &  $0.0098$ &  $0.0151$ & $0.0218$ &$0.0264$  \\
			&  RES &  9.9275e-09 & 9.9577e-09 & 9.9621e-09 & 9.9764e-09 & 9.9436e-09   \\\hline
\end{tabular}
\end{table}

\end{example}

\begin{example}\label{exam2}{\rm
Consider GAVEs \eqref{eq:gave} with the same $A$ in Example \ref{exam1} and
$$B=
\scriptsize{
\setlength{\arraycolsep}{5pt}
\renewcommand{\arraystretch}{1}
\left(
  \begin{array}{cccccccccccc}
    \bar{S_2} & -0.5I & -0.5I &-0.5I &-0.5I & 0 & 0 &0 & 0 & 0 & 0 & 0 \\
    -0.5I & \bar{S_2} & -0.5I &-0.5I &-0.5I & -0.5I & 0 &0 & 0 & 0 & 0 & 0\\
    -0.5I & -0.5I & \bar{S_2} &-0.5I &-0.5I & -0.5I & -0.5I &0 & 0 & 0 & 0 & 0\\
    -0.5I & -0.5I & -0.5I &\bar{S_2} &-0.5I & -0.5I & -0.5I &-0.5I & 0 & 0 & 0 & 0\\
    -0.5I & -0.5I & -0.5I &-0.5I &\bar{S_2} & -0.5I & -0.5I &-0.5I & -0.5I & 0 & 0 & 0\\
    0 & -0.5I & -0.5I &-0.5I &-0.5I & \bar{S_2} & -0.5I &-0.5I & -0.5I & -0.5I & 0 & 0\\
    \ddots&\ddots&\ddots&\ddots&\ddots&\ddots&\ddots
    &\ddots&\ddots&\ddots&\ddots&\ddots\\
    0 & 0&0&-0.5I & -0.5I &-0.5I &-0.5I & \bar{S_2} & -0.5I &-0.5I & -0.5I & -0.5I \\
    0 & 0&0&0&-0.5I & -0.5I &-0.5I &-0.5I & \bar{S_2} & -0.5I &-0.5I & -0.5I  \\
    0 & 0&0&0&0&-0.5I & -0.5I &-0.5I &-0.5I & \bar{S_2} & -0.5I &-0.5I   \\
    0 & 0&0&0&0&0&-0.5I & -0.5I &-0.5I &-0.5I & \bar{S_2} & -0.5I   \\
    0 & 0&0&0&0&0&0&-0.5I & -0.5I &-0.5I &-0.5I & \bar{S_2}
  \end{array}
\right)
}
\in\mathbb{R}^{n\times n},
$$
where
$$\bar{S_2}=
\scriptsize{
\setlength{\arraycolsep}{5pt}
\renewcommand{\arraystretch}{1}
\left(
  \begin{array}{cccccccccccc}
    16 & -0.5 & -0.5 &-0.5 &-0.5 & 0 & 0 &0 & 0 & 0 & 0 & 0 \\
    -0.5 & 16 & -0.5 &-0.5 &-0.5 & -0.5 & 0 &0 & 0 & 0 & 0 & 0\\
    -0.5 & -0.5 & 16 &-0.5 &-0.5 & -0.5 & -0.5 &0 & 0 & 0 & 0 & 0\\
    -0.5 & -0.5 & -0.5 &16 &-0.5 & -0.5 & --0.5 &-0.5 & 0 & 0 & 0 & 0\\
    -0.5 & -0.5 & -0.5 &-0.5 &16 & -0.5 & -0.5 &-0.5 & -0.5 & 0 & 0 & 0\\
    0 & -0.5 & -0.5 &-0.5 &-0.5 & 16 & -0.5 &-0.5 & -0.5 & -0.5 & 0 & 0\\
    \ddots&\ddots&\ddots&\ddots&\ddots&\ddots&\ddots
    &\ddots&\ddots&\ddots&\ddots&\ddots\\
    0 & 0&0&-0.5 & -0.5 &-0.5 &-0.5 & 16 & -0.5 &-0.5 & -0.5 & -0.5 \\
    0 & 0&0&0&-0.5 & -0.5 &-0.5 &-0.5 & 16 & -0.5 &-0.5 & -0.5  \\
    0 & 0&0&0&0&-0.5 & -0.5 &-0.5 &-0.5 & 16 & -0.5 &-0.5   \\
    0 & 0&0&0&0&0&-0.5 & -0.5 &-0.5 &-0.5 & 16 & -0.5   \\
    0 & 0&0&0&0&0&0&-0.5 & -0.5 &-0.5 &-0.5 & 16
  \end{array}
\right)
}
\in\mathbb{R}^{m\times m}.
$$
In addition, $b=Ax^*-B|x^*|$ with $x^*=\left(-\frac{1}{2},1,-\frac{1}{2},1,\dots, -\frac{1}{2},1\right)^\top$.

Numerical results for this example are reported in Table~\ref{table2}. It follows from  Table~\ref{table2} that the GRMS iteration is also superior to all of the other tested methods in terms of IT and CPU.}

\begin{table}[!h]\footnotesize \label{table3}
\centering
\caption{Numerical results for Example~\ref{exam2}.}\label{table2}
\begin{tabular}{ccccccc}\hline
			Method  & $m$ & $60$  & $80$ & $90$ & $100$ & $110$  \\ \hline
			\textbf{GRMS}   &  $(\theta_{\rm opt},\zeta_{\rm opt})$ & $(1.12,-0.21)$ &  $(1.13,-0.21)$  & $(1.13,-0.21)$  & $(1.13,-0.21)$ & $(1.13,-0.21)$    \\
			&  IT  &  $\textbf{18}$ & $\textbf{18}$   & $\textbf{18}$    & $\textbf{18}$  &$\textbf{18}$   \\
			&   CPU & $\textbf{0.0061}$ &  $\textbf{0.0131}$ &  $\textbf{0.0208}$ & $\textbf{0.0285}$ &$\textbf{0.0354}$  \\
&    RES &  9.9662e-09  & 9.5037e-09 & 9.6845e-09 & 9.8312e-09 & 9.9526e-09  \\
            Picard   &   &  &   &   &   &   \\
			&  IT  &  $20$ & $20$   & $19$    & $19$ & $19$   \\
			&   CPU & $0.1539$ &  $0.3426 $ &  $0.4437$ & $0.5650$ &$0.7159$  \\
			&    RES &  6.0393e-09  & 5.2093e-09 & 9.9357e-09 & 9.4355e-09 & 9.0142e-09  \\
			SOR   & $\omega_{\rm opt}$  & $0.99$ & $0.99$  &  $0.98$ &  $0.98$ & $0.98$  \\
			&  IT  &  $20$ & $20$   & $20$    & $20$ & $20$   \\
			&   CPU & $0.1552$ &  $0.3329$ &  $0.4719$ & $0.6076$ &$0.7588$  \\
			&    RES &  7.6767e-09 & 6.6559e-09 & 9.7287e-09 & 9.2968e-09 & 8.9345e-09   \\
            MFPI &   $\tau_{\rm opt}$ & $0.94$ & $0.97$  &  $0.97$ &  $0.96$ & $0.96$  \\
			&  IT  &  $21$ & $20$   & $20$    & $20$ & $20$   \\
			&   CPU & $0.1664$ &  $0.3375$ &  $0.4836$ & $0.6096$ &$0.7799$  \\
			&    RES &  8.8358e-09 & 8.9155e-09 & 8.4440e-09 & 9.6503e-09 & 9.2759e-09   \\
            SSMN&  &  &  &  &  & \\
			&  IT  &  $25$ & $25$   & $25$  &  $25$ & $25$   \\
			&  CPU & $0.1973$ &  $0.4262$ &  $0.6025$ & $ 0.7447 $ &$0.9651$   \\
			&  RES & 5.8424e-09   & 5.7611e-09  & 5.7318e-09  & 5.7075e-09  & 5.6871e-09  \\
            NSOR-like& $\alpha_{\rm opt}$ & $1.98$ & $1.98$ & $1.99$ & $1.99$ & $1.99$ \\
			&  IT  &  $45$ & $45$   & $45$  &  $45$ & $45$   \\
			&  CPU & $0.069$ &  $0.0151$ &  $0.0221$ & $0.0330$ &$0.0410$   \\
			&  RES &  9.9198e-09  & 9.9964e-09  & 9.6222e-09  & 9.6399e-09  & 9.6540e-09  \\
            MGSOR & $(\alpha_{\rm opt},\beta_{\rm opt})$ & $(1.08,~0.88)$ & $(1.02,~0.95)$ & $(1.03,~0.94)$ & $(1.04,~0.93)$ & $(1.04,~0.93)$\\
			&  IT  &  $21$ & $20$   & $20$  &  $20$ & $20$   \\
			&  CPU & $0.1694$ &  $0.3578$ &  $0.4886$ & $0.6100$ &$0.7828$  \\
			&  RES &  9.9598e-09 & 9.5198e-09 & 9.5627e-09 &  9.8469e-09 & 9.4615e-09   \\
            MAMS& $\beta_{\rm opt}$ & $1.81$ & $1.85$  & $1.86$  & $1.87$  & $1.88$   \\
			&  IT  &  $64$ & $64$   & $64$    & $64$  &$64$   \\
			&  CPU & $0.0110$ &  $0.0253$ &  $0.0354$ & $0.0506$ &$0.0604$  \\
			&  RES &  9.9610e-09 & 9.9671e-09 & 9.9809e-09 & 9.9832e-09 & 9.9774e-09   \\
            NSOR  & $(\omega_{\rm opt},\sigma_{\rm opt})$  & $(0.85,0.07)$ & $(0.85,0.07)$ & $(0.85,0.07)$ & $(0.85,0.07)$ & $(0.85,0.07)$  \\
			&  IT  &  $19$ & $19$   & $19$    & $19$  & $19$   \\
			&  CPU & $0.1490$ &  $0.3353$ &  $0.4562$ & $0.5853$ &$0.7371$  \\
			&  RES &  6.3830e-09 & 5.6182e-09 & 5.3492e-09 & 5.1280e-09 & 4.9425e-09   \\\hline
\end{tabular}
\end{table}

\end{example}

\section{Conclusions}\label{sec:Conclusions}
A generalization of the relaxed-based splitting (GRMS) iteration method for solving the generalized absolute value equations is developed. The GRMS method is a general framework which includes many existing methods as its special cases. The convergence of the proposed method is analysed and numerical results demonstrate the effectiveness of the proposed method. In addition, a few comparison theorems between GRMS and some existing methods are established. However, according to (iii) and (iv) in Section~\ref{sec:GRMS}, can we develop a method that contains the MSOR and the MAMS for solving GAVEs~\eqref{eq:gave} with $B$ being singular needs further study.


\section*{Declarations}

\begin{itemize}
\item Funding C. Chen was supported partially by the Fujian Alliance of Mathematics (Grant No. 2023SXLMQN03) and the Natural Science Foundation of Fujian Province (Grant No. 2021J01661). D. Han was supported partially by the National Natural Science Foundation of China (12131004, 11625105) and the Ministry of Science and Technology of China (2021YFA1003600).
\end{itemize}



\begin{thebibliography}{99}
\bibitem{alct2023}
J.-H. Alcantara, J.-S. Chen, M.-K. Tam. Method of alternating projections for the general absolute value equation, \emph{J. Fixed Point Theory Appl.}, 25: 39, 2023.

\bibitem{bepl1994}
A. Berman, R.J. Plemmons. Nonnegative Matrices in the Mathematical Sciences, \emph{SIAM Philadelphia}, 1994.

\bibitem{bk2004}
A. Bhaya, E. Kaszkurewicz. Steepest descent with momentum for quadratic functions is a version of the conjugate gradient method, \emph{Neural Networks}, 17: 65--71, 2004.

\bibitem{cyyh2021}
C.-R. Chen, Y.-N. Yang, D.-M. Yu, D.-R. Han. An inverse-free dynamical system for solving the absolute value equations, \emph{Appl. Numer. Math.}, 168: 170--181, 2021.

\bibitem{chyh2023}
C.-R. Chen, D.-M. Yu, D.-R. Han. Exact and inexact Douglas-Rachford splitting methods for solving large-scale sparse absolute value equations, \emph{IMA J. Numer. Anal.}, 43: 1036--1060, 2023.


\bibitem{doss2020}
X. Dong, X.-H. Shao, H.-L. Shen. A new SOR-like method for solving absolute value equations, \emph{Appl. Numer. Math.}, 156: 410--421, 2020.


\bibitem{huhu2010}
S.-L. Hu, Z.-H. Huang. A note on absolute value equations, \emph{Optim. Lett.}, 4: 417--424, 2010.

\bibitem{huli2022}
B.-H. Huang, W. Li. A modified SOR-like method for absolute value equations associated with second order cones, \emph{J. Comput. Appl. Math.}, 400: 113745, 2022.


\bibitem{jizh2013}
X.-Q. Jiang, Y. Zhang. A smoothing-type algorithm for absolute value equations,  \emph{J. Ind. Manag. Optim.}, 9: 789--798, 2013.


\bibitem{KM2017}
Y.-F. Ke, C.-F. Ma. SOR-like iteration method for solving absolute value equations,   \emph{Appl. Math. Comput.}, 311: 195--202, 2017.
	
\bibitem{ke2020}
Y.-F. Ke. The new iteration algorithm for absolute value equation, \emph{Appl. Math.  Lett.}, 99: 105990, 2020.



\bibitem{lyyhc2023}
X.-H. Li, D.-M. Yu, Y.-N. Yang, D.-R. Han, C.-R. Chen. A new fixed-time dynamical system for absolute value equations, \emph{Numer. Math. Theory Methods Appl.}, 16: 622--633, 2023.


\bibitem{LCX2016}	
C.-X. Li. A modified generalized Newton method for absolute value equations, \emph{J.  Optim. Theory Appl.}, 170: 1055--1059, 2016.
	
\bibitem{liwu2020}
C.-X. Li, S.-L. Wu. Modified SOR-like method for absolute value equations, \emph{Math. Probl. Eng.}, 2020: 1--6, 2020.

\bibitem{liyi2021}
    X. Li, X.-X. Yin. A new modified Newton-type iteration method for solving generalized absolute value equations, arXiv preprint, \href{arXiv.2103.09452}{arXiv.2103.09452}, 2021.

\bibitem{lild2022}
X. Li, Y.-X. Li, Y. Dou. Shift-splitting fixed point iteration method for solving generalized absolute value equations, \emph{Numer. Algorithms},
93: 695--710, 2023.

\bibitem{mame2006}
O.-L. Mangasarian, R.-R. Meyer. Absolute value equations, \emph{Linear Algebra Appl.}, 419:  359--367, 2006.

\bibitem{mang2007}
O.-L. Mangasarian. Absolute value programming, \emph{Comput. Optim. Appl.}, 36: 43--53, 2007.
	
\bibitem{manga2009a}
O.-L. Mangasarian. A generalized Newton method for absolute value equations,  \emph{Optim.  Lett.}, 3: 101--108, 2009.
	
\bibitem{manga2009b}
J. Rohn. On unique solvability of the absolute value equation, \emph{Optim. Lett.}, 3:  603--606, 2009.
	

\bibitem{prok2009}
O. Prokopyev. On equivalent reformulations for absolute value equations, \emph{Comput.  Optim. Appl.}, 44: 363--372, 2009.

\bibitem{prhc2006}
J.-R. Plemmons, R.-A. Horn, C.-R. Johnson. Matrix Analysis (Second edition),  \emph{Cambridge University Press}, 2013.	

\bibitem{qian1999}
  N. Qian. On the momentum term in gradient descent learning algorithms, \emph{Neural Networks}, 12: 145--151, 1999.

\bibitem{rume1986}
D.-E. Rumelhart, G.-E. Hinton, R.-J. Williams. Learning representations by back-propagating errors, \emph{Nature}, 323: 533--536, 1986.

\bibitem{rohn2004}
J. Rohn. A theorem of the alternatives for the equation $Ax + B|x| = b$, \emph{Linear  Multilinear Algebra}, 52: 421--426, 2004.	


\bibitem{rohf2014}
J. Rohn, V. Hooshyarbakhsh, R. Farhadsefat. An iterative method for solving absolute value equations and sufficient conditions for unique solvability, \emph{Optim. Lett.}, 8:  35--44, 2014.


\bibitem{soso2023}
J. Song, Y.-Z. Song. Relaxed-based matrix splitting methods for solving absolute value equations, \emph{Comput. Appl. Math.}, 42: 19, 2023.


\bibitem{wacc2019}
A. Wang, Y. Cao, J.-X. Chen. Modified Newton-type iteration methods for generalized absolute value equations, \emph{J. Optim. Theory Appl.}, 181: 216--230, 2019.


\bibitem{young1971}
D.-M. Young. Iterative Solution of Large Linear Systems, \emph{Academic Press}, New York, 1971.

	
\bibitem{yuch2022}	
D.-M. Yu, C.-R. Chen, D.-R. Han. A modified fixed point iteration method for solving the system of absolute value equations, \emph{Optimization}, 71: 449--461, 2022.


\bibitem{zhang2013}
N. Zhang. Semistability of steepest descent with momentum for quadratic functions, \emph{Neural Comput.}, 25: 1277--1301, 2013.

\bibitem{zhzl2023}
J.-L. Zhang, G.-F. Zhang, Z.-Z. Liang. A modified generalized SOR-like method for solving an absolute value equation, \emph{Linear Multilinear Algebra}, 71: 1578--1595, 2023.

\bibitem{zzll2023}
J.-L. Zhang, G.-F. Zhang, Z.-Z. Liang, L.-D. Liao. Momentum acceleration-based matrix splitting method for solving generalized absolute value equation, \emph{Comput. Appl. Math.}, 42: 300, 2023.

\bibitem{zhwl2021}	
H.-Y. Zhou, S.-L. Wu, C.-X. Li. Newton-based matrix splitting method for generalized absolute value equation, \emph{J. Comput. Appl. Math.}, 394: 113578, 2021.


\end{thebibliography}

\end{document}